\documentclass[reqno,a4paper,11pt]{amsart}

\usepackage[english]{babel}
\usepackage{latexsym,amsmath,amsthm,bm}
\usepackage{amsfonts,xcolor}
\usepackage{graphicx}
\usepackage[colorlinks=true, allcolors=blue]{hyperref}
\usepackage{comment}
\usepackage{fullpage,hyperref,cleveref,amssymb,mathtools,ifthen,thmtools,tikz,bbm}
\usepackage{thm-restate}
\usepackage{algorithm}

\usepackage[backend=biber,giveninits=true,doi=false,url=false,isbn=false,style=trad-plain,hyperref,date=year]{biblatex}
\bibliography{references}

\newtheorem{theorem}{Theorem}[section]
\newtheorem{definition}[theorem]{Definition}
\newtheorem{proposition}[theorem]{Proposition}

\newtheorem{lemma}[theorem]{Lemma}

\newtheorem{question}[theorem]{Question}

\newtheorem{remark}[theorem]{Remark}

\newcommand{\cA}{\mathcal{A}}

\newcommand{\cC}{\mathcal{C}}
\newcommand{\cD}{\mathcal{D}}
\newcommand{\cE}{\mathcal{E}}

\newcommand{\cG}{\mathcal{G}}
\newcommand{\cH}{\mathcal{H}}
\newcommand{\cI}{\mathcal{I}}
\newcommand{\cJ}{\mathcal{J}}

\newcommand{\cL}{\mathcal{L}}

\newcommand{\cO}{\mathcal{O}}
\newcommand{\cP}{\mathcal{P}}

\newcommand{\cS}{\mathcal{S}}

\newcommand{\cV}{\mathcal{V}}
\newcommand{\cW}{\mathcal{W}}

\newcommand{\bI}{\mathbf{I}}

\newcommand{\ta}{\tilde{\alpha}}

\newcommand{\lam}{\lambda}

\newcommand{\om}{\omega}

\newcommand{\Gam}{\Gamma}
\newcommand{\Om}{\Omega}

\renewcommand{\Pr}{\mathbb{P}}
\newcommand{\E}{\mathbb{E}}

\newcommand{\tmix}{\tau_{\rm{mix}}}

\providecommand\given{\nonscript\:\ifthenelse{\equal{\delimsize}{}}{\big\vert}{\delimsize\vert}\nonscript\:\mathopen{}}
\let\Pr\undefined
\DeclarePairedDelimiterXPP\Pr[1]{\mathbb{P}}(){}{#1}
\DeclarePairedDelimiterXPP\Ex[1]{\mathbb{E}}{[}{]}{}{#1}
\DeclarePairedDelimiter{\norm}{\lVert}{\rVert}

\newcommand{\ex}{\kappa}
\newcommand{\exx}{C_*}
\newcommand{\codeg}{\Delta_2}

\title{Sampling from the antiferromagnetic Ising model on bipartite, regular expander graphs}
\author{Anna Geisler, Mihyun Kang, Michail Sarantis, Ronen Wdowinski}

\address{Institute of Discrete Mathematics, Graz University of Technology, Steyrergasse 30, 8010 Graz, Austria}

\email{\{geisler,kang,sarantis,wdowinski\}@math.tugraz.at}

\begin{document}

\maketitle
\begin{abstract}
    The antiferromagnetic Ising model samples subsets of vertices of a graph with weight decaying exponentially in the number of edges induced. We study the problem of sampling from this model on the class of bipartite, regular graphs with good vertex expansion. We show that a natural sampler, namely the Glauber dynamics, mixes exponentially slowly in a wide range of parameters. On the other hand, we give an efficient alternative algorithm for sampling from the Ising model and an FPTAS for its partition function, using polymer models and the cluster expansion method.
\end{abstract}

\section{Introduction}

\subsection{Motivation and Background}

The \emph{Ising model} is a 2-spin model that was introduced in 1920 as a mathematical model of ferromagnetism and has since gained a lot interest in both statistical physics and mathematics. See \cite{Cipra87} and \cite[Chapter 3]{FriVe18} for an introduction.
Given an underlying graph structure, assign a spin from $\{-1,1\}$ to each vertex. The weight of a spin configuration is determined by the number of edges where both endpoints have the same spin. In this paper, we study a related model that is equivalent to the Ising model on regular graphs, referring to it as the Ising model throughout. We give a proof of the equivalence of the models in \Cref{sec:ising}. See \cite[Section 2.1]{SuSly12} for a more general result on the equivalence of $2$-spin systems.

Given a graph $G=(V, E)$, a \textit{spin configuration} is a map $\sigma: V \to \{0, 1\}$. Equivalently, it may be viewed as a subset of vertices $S \coloneqq \{v\in V:\sigma(v)=1\}$. Given two parameters, the \textit{fugacity} $\lambda >0$ and the \textit{inverse temperature} $\beta \in \mathbb{R}$, we define the weight of a subset $S \subseteq V$ as
\[
\tilde{\omega}(S) \coloneqq \lambda^{|S|} e^{- \beta |E(S)|},
\]
where $E(S)$ is the set of edges in $G$ induced by $S$. The \textit{Ising model} samples a subset $S \subseteq V$ with probability proportional to its weight $\tilde{\omega}(S)$. That is,
\begin{align} \label{eq:Ising_distribution}
\mathbb{P}(\mathbf{S}=S)=\frac{\tilde{\omega}(S)}{Z_G(\lambda, \beta)},
\end{align}
where $Z_G(\lam,\beta)$ is the \emph{partition function} of the Ising model, given by
\[
Z_G(\lambda, \beta) \coloneqq \sum_{S \subseteq V} \lambda^{|S|} e^{- \beta |E(S)|}.
\]
We denote by $\mu_{\lam,\beta}$ the derived probability measure, the so-called \textit{Gibbs measure}. We say that the model is \emph{ferromagnetic} if $\beta<0$, and \textit{antiferromagnetic} if $\beta>0$. The weight of a set $S$ increases with the number of edges it induces in the ferromagnetic setting, while the weight decreases in the antiferromagnetic setting. This paper focuses on the antiferromagnetic (in short, AF) setting, that is, the regime $\beta \in (0, \infty)$.

One reason we consider this variation of the classical Ising model is its combinatorial interpretation as a positive temperature analogue of the \textit{hard-core model}. Recall that a subset $I \subseteq V$ is called \textit{independent} if it induces no edges. The hard-core model with fugacity $\lam>0$ samples an independent set $I \subseteq V$ of $G$ with probability
$$\mathbb{P}(\bI=I)=\frac{\lam^{|I|}}{Z_G(\lam)},$$
where $Z_G(\lam)$ is the partition function, also called the \textit{independence polynomial}, given by
\begin{align*}
    Z_G(\lam) \coloneqq \sum_{I\in\cI(G)}\lam^{|I|}.
\end{align*}
Hence, informally, the hard-core model is the Ising model at the zero temperature limit $\beta=\infty$. A priori, we cannot immediately pass results from one model to the other by setting $\beta=\infty$, but in practice that will not be difficult by slightly modifying the proofs.

In fact, there is a deeper connection of the AF Ising model to the hard-core model beyond being an informal positive temperature analogue. In \cite{kronenberg2022independent}, Kronenberg and Spinka observed that if we randomly sparsify a graph, the expectation of the partition function of the hard-core model on the sparsified graph is precisely the partition function of the Ising model on the original graph (after appropriate parametrization). More precisely, given $p \in [0,1]$, let $G_p$ be the random graph (or percolated graph) obtained by retaining each edge of $G$ independently with probability $p$, and define $\beta > 0$ so that $p=1-e^{-\beta}$. It was observed in \cite{kronenberg2022independent} that
\begin{equation*}
    \E[Z_{G_p}(\lam)]=Z_G(\lam,\beta).
\end{equation*}
This ability to transfer from such a percolated hard-core model to the AF Ising model was essential for the results of \cite{kronenberg2022independent} about counting independent sets in the percolated hypercube $Q^d$. This work was further expanded upon by this paper's authors in \cite{GeKaSaWdo25}. See \cite{balogh2021independent,chowdhury2024gaussian, CEGK25,jenssen2020independent} for related work on counting independent sets.

The goal of this paper is to study how to efficiently sample from the AF Ising model, and our results are two-fold. First, we exclude a well-known sampler, the \textit{Glauber dynamics}, as an appropriate sampling scheme for the AF Ising model in a broad family of regular graphs with good expansion properties. Secondly, we show that, for the same family of graphs, there is a different algorithm which yields a polynomial time sampling scheme for the Ising model in bounded degree graphs and an FPTAS for its partition function. This generalizes results of \cite{GaTe2006} and \cite{JePePo23} for the hard-core model, and at the same time provides the best up-to-date range of parameters for the validity of these results. Finally, we propose a modification of the Glauber dynamics for graphs that exhibit certain symmetries, and leave as an open question whether this new Markov chain is rapidly mixing.

\subsubsection{Approximation Algorithms}
The search for efficient algorithms to compute the partition function of the Ising model or sample from the corresponding distribution has received great attention. Exact evaluation of the partition function of the \emph{ferromagnetic} Ising model is $\#$P-hard \cite{JeSi93}, while the \emph{antiferromagnetic} case might be even harder \cite{HuPePo24}. The reason that the AF case is computationally hard lies in its close relation to the Max-Cut problem \cite{CoLoMeSo22,DyHeJeMu21}. Thus, it is reasonable to pursue approximation algorithms. In this paper, we focus on approximation algorithms for the AF Ising model on regular, bipartite graphs that exhibit good vertex-expansion.

One of the main approaches to approximately sample from a distribution $\pi$ is the \textit{Markov Chain Monte Carlo method} (MCMC). The underlying idea is to run a Markov chain with stationary distribution $\pi$. Given mild conditions on the chain, the stationary distribution is unique and the distribution of the output after time $t$ tends to $\pi$ as $t\rightarrow\infty$. Whether this is a useful algorithm depends on the \textit{mixing time}, that is, the time it takes to get close enough to $\pi$ in some metric, e.g., in total variation distance.

In this direction, Markov chains with local updates are well-studied and are preferred candidates due to their simple update rules and the computational efficiency of their implementation. One of the most celebrated and well-studied local chains is the \textit{Glauber dynamics} $M(G)$. In the case of the AF Ising model, we start with any subset $S_0 \subseteq V$. In each step, we sample a vertex $v \in V$ uniformly at random and update our set to
\begin{align*}
    S_{t+1} \coloneqq \left\{\begin{array} {c@{\quad \textup{with probability} \quad}l}
       S_t \cup \{v\} & \dfrac{\lambda e^{-\beta |N(v) \cap S_t|}}{1+\lambda e^{-\beta |N(v) \cap S_t|}}, \\[+1ex]
      S_t \setminus \{v\} & \dfrac{1}{1+\lambda e^{-\beta |N(v) \cap S_t|}}.
    \end{array}\right.
\end{align*}
Note that the update probabilities are given by the probability that $v \in S_{t+1}$ conditional on the state of all other vertices. It is easy to show that this Markov chain is ergodic, aperiodic, and time-reversible, and that its (unique) stationary distribution is indeed given by $\pi = \mu_{\lambda, \beta}$, the Gibbs measure defined by \eqref{eq:Ising_distribution}. In the realm of positive results, Chen, Feng, Yin and Zhang \cite{CheFeZhaYi22} used spectral independence techniques to prove that the Glauber dynamics on AF $2$-spin systems mix in time $C(\delta)n\log n$ for all graphs (even with unbounded degree), when the parameters satisfy the uniqueness condition with slack $\delta\in(0,1)$ (see \cite[Definition 3.1]{CheFeZhaYi22}). One can verify that for our model, this implies fast mixing when $\lam(1-e^{-\beta})\lesssim 1/d$; and for the hard-core model when $\lam\lesssim 1/d$.
We note that very recent work of Chen and Jiang \cite{Chen2025Improved} broke this barrier for rapid mixing for the hard-core model, by a small parameter depending on $n$. Specifically, for graphs with $n$ vertices and maximum degree $\Delta$ (possibly depending on $n$), if $\gamma > 0$ is a constant and $\lam\leq\left(1+\frac{\gamma}{\sqrt{n}}\right)\lam_c(\Delta)$, where $\lam_c(\Delta)$ is the uniqueness threshold, then the Glauber dynamics mixes in time $O_\gamma\left(n^{2+2e+\frac{2e}{\Delta-2}}\log\Delta\right).$ 

Closely related is the approximate counting problem, that is, approximating the partition function. The sampling schemes mentioned yield an FPRAS for the partition function in the same regime by a standard self-reducibility argument. 
First we note that there exist \textit{deterministic} algorithms for the partition function of AF $2$-spin systems up to the uniqueness threshold that run in polynomial time, but with exponential dependence on the degree of the graph \cite{li2013correlation}. On the other hand, computing the partition function of AF $2$-spin systems even approximately is NP-hard beyond the uniqueness threshold for $d$-regular graphs \cite{SuSly12}.

However, this does not exclude the possibility of efficient algorithms in other subclasses. An important, and still unresolved, related complexity class is that of \#BIS-hard problems: problems that are polynomially equivalent to approximating the number of independent sets of a bipartite graph. It has been shown in \cite{cai2016hardness} that approximating the partition function of the hard-core model in bipartite graphs beyond the uniqueness threshold is a \#BIS-hard problem. Other problems that have been shown to be \#BIS-hard can be found in \cite{chebolu2012complexity,galanis2016approximately,goldberg2012approximating}.

The fact that \#BIS problems have resisted any attempt for an FPRAS had led Goldberg and Jerrum \cite{goldberg2012approximating} to conjecture that there is no FPRAS for \#BIS problems. Thus far, positive results have only been shown in certain subfamilies of bipartite graphs. For bipartite expanders, there exist FPTAS for some range above the uniqueness threshold, but they have an exponential dependence on the degree \cite{JePePo23}. In the same class, there also exist FPRAS for unbounded degree graphs but they require the fugacity $\lam$ to grow polynomially in the degree \cite{CheGaGoPeSteVi21}. Furthermore, Jenssen, Keevash, and Perkins \cite{JeKePe20} proved that, in general, hard instances are rare (if they exist): if $\Delta\geq\Delta_0$ for some constant $\Delta_0$, and $\lam\geq \frac{50\log^2\Delta}{\Delta}$, then there exists an FPTAS that approximates the independence polynomial of almost all graphs of maximum degree $\Delta$, even when $\Delta$ is unbounded.

Even for the celebrated instance of the hypercube $Q^d$, it is unknown whether a polynomial-time approximate uniform sampler or an FPRAS for the number of independent sets exists. Applying the existing results on $Q^d$, which has $n=2^d$ vertices, an approximate sampler and an FPRAS up to precision $e^{-n}$ for the partition function exist when $\lam\geq d^{\Om(d)}=n^{\Om(\log\log n)}$ \cite[Theorem 11]{CheGaGoPeSteVi21}, while deterministic algorithms for $\lam\geq \frac{C_0 \log^{3/2}d}{d^{1/2}}$ run in time $n^{\text{poly}(d)}=n^{\text{polylog}(n)}$ (\cite{jenssen2024refined,JePePo23} and this work). Consequently, we ask whether this happens due to limitations of current methods or if the problem does not admit an FPRAS. In \cite{GaTe2006}, Galvin and Tetali proved that the Glauber dynamics for the hard-core model mix exponentially slowly by identifying a natural bottleneck. Thus, a potential candidate for an approximate sampler could be a modification that overcomes the bottleneck. We discuss this in more detail in \Cref{sec:discussion}.

\subsection{Main results}

For definitions not yet specified, see \Cref{sec:prelim}. We define a class of bipartite, regular graphs with bounded codegree satisfying a weak global vertex-expansion property. This will be the setting in which we prove our slow mixing result.

\begin{definition}\label{def:H}
    The class $\cH=\cH(n, d, \codeg, \ex)$ is the set of $n$-vertex, $d$-regular, bipartite graphs $G$ with bipartition $(\cO, \cE)$ satisfying the following three properties:
    \begin{itemize}
        \item[(1)] all pairs of vertices of $G$ have at most $\codeg$ common neighbors,
        \item[(2)] $|N(X)| \geq \left(1+\Omega\left(\frac{1}{d^{\ex}}\right)\right)|X|$ for all $X \subseteq \cO, \cE$ whenever $|X| \leq \frac{3}{8}n$, and
        \item[(3)] $n=\Om(d^6\log d)$.
    \end{itemize}
\end{definition}

Next, for our result on more efficient sampling algorithms, we will work with the following subclass of $\cH(n, d, \codeg, \ex)$ that puts a stronger control on the vertex-expansion for sets of polynomial size. 

\begin{definition}\cite[Definition 11.1]{GeKaSaWdo25}\label{def:H'}
The class $\cH'=\cH'(n,d,\codeg,\ex)$ is the set of all graphs from $\cH(n, d, \codeg, \ex)$ that satisfy $(1), (2), (3)$ of \Cref{def:H} and the following additional condition:
    \begin{itemize}
        \item[(4)] $|N(X)| \geq \sqrt{d} |X|$ for all $X \subseteq \cO, \cE$ whenever $|X| \leq d^3 \log n$.
    \end{itemize}  
\end{definition}

Condition (4) is required for the use of cluster expansion methods, allowing us to apply results from our previous work \cite{GeKaSaWdo25} about counting independent sets in a percolated graph from the class $\cH'$.

Given $0 \le \ex < 2$, it is convenient for us to set
\begin{align*}
    \exx \coloneqq \min\left\{\frac{1}{2}, 1-\frac{\ex}{2}\right\}.
\end{align*}
Our main results consider the parameters $\lam, \beta > 0$ under the conditions that
\begin{align} \label{eq:cond_lam_beta}
    \lambda \leq \lam_0 \qquad \text{ and } \qquad  \lam(1 - e^{-\beta}) \geq \frac{C_0 \log^{3/2} d}{d^{\exx}},
\end{align}
where $\lambda_0 > 0$ is a fixed constant and $C_0$ is a sufficiently large constant depending only on $\lambda_0$. In our main theorems, we do not expect to improve these conditions beyond $\lam (1 - e^{-\beta}) = \tilde{\Omega}(1/d)$, as it is believed this is where the typical structure of independent sets drawn according to the hard-core model, i.e., when $\beta=\infty$, exhibits a phase transition (see \cite{galvin2011threshold, jenssen2024refined}). For convenience throughout this paper, we set
$$\alpha \coloneqq \lambda (1 - e^{-\beta}).$$

Our first main result states that for graphs in the class $\cH(n, d, \codeg, \ex)$, under the above conditions on $\lam, \beta$ the Glauber dynamics mix exponentially slowly.
Our bound on the mixing time $\tmix(M(G))$ of the Glauber dynamics $M(G) = M_{\lam,\beta}(G)$ on $G$ will depend on the values of $\lam, \beta > 0$ satisfying the condition \eqref{eq:cond_lam_beta}, as well as the parameters $n, d, \ex$ of the graph $G$. We set
\begin{align} \label{eq:def-flambda}
    f(\lam) \coloneqq \frac{\lam}{R \log^2 (\max\{e, \lam^{-1}\})},
\end{align}
where $R$ is a sufficiently large constant depending only on $\lam_0$. Note that $f(\lam)$ is constant whenever $\lam$ is constant, and that $f(\lam) \geq \frac{2C_0}{R d^{1/2} \log^{1/2} d}$ because $\lam \geq \frac{C_0 \log^{3/2} d}{d^{1/2}}$.
\begin{theorem} \label{thm:slowmixing}
    Fix $\codeg \ge 1$, $0 \le \ex < 2$, and $\lam_0>0$. Then there exist constants $C_0, C > 0$ such that whenever $\lambda \le \lam_0$ and $\alpha \coloneqq \lam(1 - e^{-\beta}) \ge \frac{C_0 \log^{3/2} d}{d^{\exx}}$, the mixing time of the Glauber dynamics $M(G)$ on a graph $G \in \cH(n,d,\codeg,\ex)$ satisfies
    \[
    \tmix(M(G)) \geq \exp\left( \frac{C \alpha^2 f(\lam)}{d^{\ex} \log d} \cdot n\right).
    \]
\end{theorem}

The proof of \Cref{thm:slowmixing} is given in \Cref{sec:slowmixing}. Taking the zero temperature limit $\beta = \infty$ in Theorem \ref{thm:slowmixing}, we obtain a slow mixing result for the hard-core model initially proven in \cite{GaTe2006}, for regular bipartite expanders when $\lam = {\Omega}\left(\frac{\log^{3/2} d}{d^{1/4}}\right)$. A bound similar to ours, for the hard-core model, was also recently proven by Jenssen, Malekshahian, and Park \cite{jenssen2024refined}, and thus our result provides an extension to the AF Ising model.

While Theorem \ref{thm:slowmixing} shows that the Glauber dynamics is an inefficient sampling algorithm, our second main result gives a polynomial-time approximation scheme for sampling according to the Ising model on the subclass $\cH'(n,d,\codeg, \ex)$.
\begin{theorem} \label{thm:efficient_sampling}
    Fix $\codeg \ge 1$, $0 \le \ex < 2$, and $\lam_0>0$. Then there exist constants $C_0, d_0 > 0$ such that whenever $d \ge d_0$, $\lambda \le \lam_0$ and $\alpha \coloneqq \lam(1 - e^{-\beta}) \ge \frac{C_0 \log^{3/2} d}{d^{\exx}}$, there exists a polynomial-time approximation scheme for sampling from the Ising model $\mu_{\lam,\beta}$ and an FPTAS for its partition function on $\cH'(n,d,\codeg,\ex)$.
\end{theorem}
The algorithm outputs a sample in runtime $\left(\frac{n}{\epsilon} \right)^{O(d^6 \log d)}$, where $\epsilon > 0$ is the tolerance on the total variation distance of the sampler to the Gibbs measure $\mu_{\lam,\beta}$. The deterministic approximation of $Z_G(\lam,\beta)$ is computed as part of the construction of the sampler, and it has the same running time.
The algorithm and the proof for \Cref{thm:efficient_sampling} are given in \Cref{sec:sampling}. We build on work of Jenssen, Perkins, and Potukuchi \cite{JePePo23}, who gave an FPTAS for the hard-core model on regular bipartite expanders for constant $d$ when $\lambda = \Omega\left(\frac{\log d}{d^{1/4}}\right)$, which was subsequently improved to $\lambda = \Omega\left(\frac{\log^2 d}{d^{1/2}}\right)$ by Jenssen, Malekshahian, and Park \cite{jenssen2024refined}. We note that their results for this range of $\lam$ are only formulated for constant expansion, that is, $\ex=0$.
As with their work, our sampling scheme relies on the development by Jenssen and Perkins \cite{jenssen2020independent} of the cluster expansion method for bipartite expanders, which provides a rigorous and efficient way to describe the \emph{typical} structure of a set drawn from the hard-core or the Ising model. Using the theory of polymer models and cluster expansion, we reduce the task of approximating the partition function to the computation of \emph{clusters} of bounded size and their weights, which may be done efficiently. 

While our work follows similar approaches to analogous work on the hard-core model in the references mentioned above \cite{GaTe2006, jenssen2024refined, JePePo23}, parts of our proofs require additional technical considerations that do not arise in the hard-core model, since we need to consider more than just independent sets. We handle these complications using results and techniques from our previous paper \cite{GeKaSaWdo25}, which follow similar techniques to \cite{kronenberg2022independent,peled2020long}.

\begin{remark}
    The two classes of graphs in \Cref{def:H} and \Cref{def:H'} include a variety of commonly studied graphs, for example,
    \begin{itemize}
        \item[(a)] the hypercube $Q^d$, where $\ex = 1/2$ (see \cite{galvin2019independent});
        \item[(b)] the middle layer of the hypercube $Q^d$ for odd $d$, where $\ex = 1$ (see \cite{balogh2021independent});
        \item[(c)] (growing) even tori $\mathbb{Z}_m^t$, for even $m=O(t^c)$ and any $0<c<3/2$, where $\ex = c+\frac{1}{2}$ (see \cite{CEGK25, GeKaSaWdo25});
        \item[(d)] Cartesian products of bounded-size, regular bipartite graphs, where $\ex = 1/2$ (see \cite{CEGK25}).
    \end{itemize}
    The graphs in (a)--(c) have maximum codegree $\codeg = 2$ when $d \ge 2$. For graphs in (d), the maximum codegree is the maximum of $2$ and the maximum codegree of any base graph, and since the base graphs are of bounded size, this is bounded by a constant. We note that the slow mixing result of \Cref{thm:slowmixing} is valid for these graphs. But as mentioned before, \Cref{thm:efficient_sampling} only gives a quasipolynomial time sampling scheme due to the exponential dependence on the degree $d$. In \Cref{sec:discussion}, we propose a different Markov chain that could potentially mix fast for symmetric graphs like these. 
\end{remark}

\subsection{Structure of the paper}
After basic definitions and notation, \Cref{sec:prelim} covers basic computational concepts as well as a short introduction to polymer models and the cluster expansion. \Cref{sec:slowmixing} is dedicated to the proof of our first main result, \Cref{thm:slowmixing}, after an introductory discussion. The proof relies on results collected in \Cref{lem:small}, \Cref{lem:middle}, \Cref{lem:rest} whose proofs are deferred to subsequent subsections. In particular, \Cref{sec:container} contains the crucial part of the argument given by a container lemma (\Cref{lem:container_main}).

Our second main result, \Cref{thm:efficient_sampling}, is proven in \Cref{sec:sampling}. After setting up the corresponding polymer model, the two main ingredients are the efficient computation of clusters of bounded size and their weights (\Cref{lem:compute_L}), and an efficient sampling algorithm for polymer configurations (\Cref{lem:compute_nu}). The proofs of these lemmas are deferred to \Cref{sec:sample_small_cluster} and \Cref{sec:compute_nu}, respectively. We conclude with a short discussion of the results and possible avenues for future research in \Cref{sec:discussion}.

\section{Preliminaries} \label{sec:prelim}

\subsection{Definitions and basic facts}
Throughout the paper, whenever we consider an $n$-vertex $d$-regular graph satisfying property $\mathcal{C}$, we in fact consider a sequence $(d_{\ell}, n_{\ell})_{\ell \in\mathbb{N}}$ of pairs $(d_{\ell},n_{\ell})\in \mathbb{N}^2$ giving rise to a sequence $(G_{\ell})_{\ell \in\mathbb{N}}$ of $n_{\ell}$-vertex $d_{\ell}$-regular graphs such that there exists $\ell_0 \in\mathbb{N}$ for which all graphs $G_{\ell}$ satisfy property $\mathcal{C}$ whenever $\ell \geq \ell_0$. All asymptotics are with respect to $d \to \infty$ and we use standard Landau notation for asymptotics (e.g., $O(\cdot), o(\cdot), \Omega(\cdot), \Theta(\cdot)$).
For ease of presentation, we will omit floor/ceiling signs and assume that $n$ is even whenever necessary throughout the paper. Unless stated otherwise, all logarithms have the natural base.

Let $G=(V,E)$ be a graph. Given a set $A \subseteq V$, we denote the \emph{(external) neighborhood} of $A$ by $N_G(A) \coloneqq \{v \in V \setminus A : \text{ there exists } w \in A \text{ with } vw \in E\}$, and we write $N_G(v)$ for $N_G(\{v\})$. For a vertex $v \in V$, we denote by $d_G(v)$ the \emph{degree} of $v$, i.e., $d_G(v)\coloneqq|N_G(v)|$. Given two distinct vertices $u,v\in V$, we define their \textit{codegree} to be $\operatorname{codeg}(u,v)\coloneqq |N_G(u)\cap N_G(v)|$.  
The \textit{maximum codegree} of $G$ is defined as $\max\{\operatorname{codeg}(u,v):\ u,v\in V, u\neq v\}$. 
For two disjoint subsets $A, B\subseteq V$, we denote by $E_G(A, B)$ the set of edges with one endpoint in $A$ and one endpoint in $B$, and we denote by $E_G(A)$ the set of edges with both endpoints in $A$.
Whenever the graph $G$ is clear from context, we suppress the subscript $G$ in this notation.

We primarily study bipartite graphs $G$, and we usually denote the bipartition by $(\cO, \cE)$, where $\cO$ and $\cE$ are viewed as the ``odd side" and ``even side" respectively, analogous to the standard bipartition of the hypercube $Q^d$. Let $\cD \in \{\cO, \cE\}$. The (bipartite) \textit{closure} of a subset $A\subseteq \cD$ is defined as
$$[A]\coloneqq\{v\in \cD:N(v)\subseteq N(A)\}.$$ 
A vertex subset $A \subseteq V$ is said to be \textit{$2$-linked} if $A$ induces a connected subgraph of $G^2$, where the graph $G^2$ has the same vertex set as $G$ and $uv$ is an edge in $G^2$ whenever the distance between $u$ and $v$ in $G$ is at most $2$. The following lemma is well known and can be found in, e.g., \cite{GaKa2004}.
\begin{lemma}[\cite{GaKa2004}] \label{l:counting2linked}
  Let $\ell \in \mathbb{N}$ and $G$ be a $d$-regular graph. For any vertex $v$ of $G$, the number of $2$-linked subsets of size $\ell$ that contain $v$ is at most $(ed^2)^{\ell-1}$.
\end{lemma}
The following estimate on the binomial coefficients will also be used:
\begin{equation}\label{eqn:binom_estimate}
    \binom{n}{k}\leq \binom{n}{\leq k} \leq \exp\left(k\log\frac{en}{k}\right),
\end{equation}
where $\binom{n}{\leq k} \coloneqq \sum_{j=0}^k \binom{n}{j}$.

\subsection{Computational aspects}
Given two probability measures $\mu$ and $\hat{\mu}$ on a sample space $\Omega$, the \emph{total variation distance} between them is defined by
\begin{equation}\label{eq:totalvariationdist}
  \norm{\mu-\hat{\mu}}_{TV} \coloneqq \frac{1}{2} \sum_{x \in \Om} \left|\mu(x) - \hat{\mu}(x)\right| = \max_{S \subseteq \Om} |\mu(S) - \hat{\mu}(S)|.
\end{equation}
Let $G = (V, E)$ be a graph, and let $\mu$ be a probability distribution on $2^V$, the set of all subsets of $V$. A \emph{polynomial-time sampling algorithm} for $\mu$ is a randomized algorithm such that, for every $\epsilon>0$, an output $S \subseteq V$ is computed in time polynomial in $|V|, \frac{1}{\epsilon}$, and the distribution $\hat{\mu}$ of such outputs $S$ satisfies
\begin{align*}
    \lVert\mu - \hat{\mu}\rVert_{TV} \leq \epsilon.
\end{align*}
A \emph{fully-polynomial time approximation scheme} (FPTAS) for a function $Z=Z(G) : G \rightarrow \mathbb{R}$ is an algorithm such that, for every $\epsilon>0$, the output $\hat{Z} = \hat{Z}(G)$ is computed in time polynomial in $|V|, \log\frac{1}{\epsilon}$, and it is an $\epsilon$-approximation of $Z$, i.e.,
\begin{align*}
(1-\epsilon) \hat{Z} \leq Z \leq (1+\epsilon) \hat{Z}.
\end{align*}

\subsection{Markov chains and mixing times}
We recall some definitions from the theory of Markov chains. See \cite{levin2017markov} for more background on this topic.

Let $M$ be a Markov chain on a state space $\Om$ with transition matrix $P$ and stationary distribution $\pi$, meaning that $\pi P=\pi$. For the sake of brevity, we may refer to a Markov chain $M$ by its transition matrix $P$ directly. A chain is called \textit{irreducible} if any state can be reached by any other state in a finite amount of time, i.e., for every $x,y\in \Omega$ there exists an integer $t$ such that $P^t(x,y)>0$. A Markov chain is called \textit{aperiodic} if all its states $x$ satisfy $\gcd(\{t:P^t(x,x)>0\})=1$. If $M$ is irreducible and aperiodic, then it is called \textit{ergodic}. It is a well-known fact that every ergodic chain has a unique stationary distribution. Furthermore, a Markov chain is called \textit{reversible} if it satisfies the \textit{detailed balance equations}:
\begin{equation*}
    \pi(x)P(x,y)=\pi(y)P(y,x) \qquad \forall x,y\in \Om,
\end{equation*}
i.e., the matrix $P$ is self-adjoint with respect to $\pi$.

Given any $\epsilon>0$, the \emph{mixing time} of $M$ is defined as
\[
\tau_{\rm{mix}}(M,\epsilon) \coloneqq \min\left\{t:\norm{P^t(x,\cdot)-\pi}_{\rm{TV}}\leq\epsilon \ \forall x\in\Om\right\}.
\]
It is standard to restrict attention to $\epsilon = 1/4$, and we denote
\[
\tmix(M)\coloneqq \tmix(M,1/4).
\]
When we refer to the mixing time without specifying $\epsilon$, we are referring to $\tmix(M)$. There is a vast literature of techniques and tools to estimate mixing times of Markov chains, and we refer the interested reader to \cite{levin2017markov} and the references therein.

For any states $x,y\in \Om$ and any subsets $A,B\subseteq\Om$, denote
\begin{equation*}
    Q(x,y):=\pi(x)P(x,y) \qquad \text{and} \qquad Q(A,B):=\sum_{x\in A, y\in B}Q(x,y).
\end{equation*}
The \textit{conductance} of a set $S \subseteq \Om$ is defined as
$$\Phi(S) = \Phi_M(S) :=\frac{Q(S,S^c)}{\pi(S)}.$$
The conductance of the chain $M$ is then given by
\begin{equation*}
    \Phi = \Phi_M:=\min_{S \subseteq \Om\, :\, 0<\pi(S)\leq\frac{1}{2}}\Phi(S).
\end{equation*}

The following well-known lemma bounds the mixing time of Markov chain $M$ in terms of the conductance $\Phi$ (see, e.g., \cite[Theorem 7.4]{levin2017markov}). 

\begin{lemma} \label{lem:conduc_mix}
If $M$ is an ergodic Markov chain with conductance $\Phi$, then its mixing time satisfies
    \begin{equation*}
        \tmix(M)\geq\frac{1}{4\Phi}.
    \end{equation*}
\end{lemma}

In particular, if we can find a set $S \subseteq \Om$ with small conductance, then we get that the Markov chain $M$ has a large mixing time.

\subsection{Polymer models and the cluster expansion}

The following is terminology on two tools from statistical physics that we will use, namely polymer models and the cluster expansion. For more on this topic, see \cite{jenssen2020independent, kotecky1986cluster}.

Let $\cP$ be a finite set of objects, which we call \textit{polymers}, equipped with a \textit{weight function} $\om:\cP\rightarrow [0,\infty)$. We also equip $\cP$ with a symmetric, anti-reflexive relation $\sim$, and we say that two polymers $A,A'\in \cP$ are \textit{compatible} if $A\sim A'$, and \textit{incompatible} otherwise (denoted by $A\nsim A'$). We call a set of pairwise compatible polymers a \textit{polymer configuration}, and we let $\Om(\cP)$ denote the collection of polymer configurations with polymers from $\cP$. (Not to be confused with the notation $\Om$ that we used to denote the state space of a Markov chain $M$.) The triple $(\cP,\om,\sim)$ is called a \textit{polymer model} and we define its \textit{partition function} to be
$$\Xi(\cP)\coloneqq\sum_{\Theta\in\Om(\cP)}\prod_{A\in \Theta} \om(A).$$

For an ordered  tuple $\Gamma$ of polymers, the \emph{incompatibility graph} $H(\Gamma)$ is the graph with vertex set $\Gamma$, in which there is an edge between any two elements of $\Gamma$ that are incompatible.
A \textit{cluster} is an ordered tuple of polymers for which the incompatibility graph is connected. We denote by $\cC(\cP)$ the collection of all clusters with polymers from $\cP$. We extend the weight function to clusters by setting
$$\om(\Gam)\coloneqq\phi(H(\Gam))\prod_{A\in \Gam}\om(A),$$
where $\phi$ is the \textit{Ursell function} of a graph $G=(V, E)$, defined by
$$\phi(G)\coloneqq\frac{1}{|V|!}\sum_{\substack{F\subseteq E\\ \text{spanning, connected}}}(-1)^{|F|}.$$
The \textit{cluster expansion} is the formal power series for $\log\Xi(\cP)$ given by
\begin{equation}\label{eqn:cluster_exp}
    \log\Xi(\cP)=\sum_{\Gam\in\cC(\cP)}\om(\Gam).
\end{equation}
Note that this is indeed an infinite formal series: Although the set of polymers is finite, the set of clusters is infinite, as there exist arbitrarily long sequences of incompatible polymers (e.g., by repeating the same polymer any number of times). In fact, this is the multivariate Taylor expansion of $\log \Xi_\cP$ in the variables $\om(A)$, and equality \eqref{eqn:cluster_exp} was observed by Dobrushin \cite{dobrushin1996estimates}. For the cluster expansion to be a useful approximation of the partition function, we need to know that it converges absolutely, and this is generally verified using the well-known \textit{Koteck\'y--Preiss condition} \cite{kotecky1986cluster} (see also \cite{jenssen2020independent}). 

The following notation will be useful when working with the cluster expansion. The {\em size} of a cluster $\Gamma\in \cC(\cP)$ is defined as $\norm{\Gamma} \coloneqq \sum_{A \in \Gamma} |A|$. For each $k \in \mathbb{N}$, define the sets
$\cC_{k}\coloneqq \{\Gamma\in \cC(\cP): \norm{\Gamma} =k\}$, $\cC_{\leq k}\coloneqq \bigcup_{j\leq k}\cC_{j}$, and  $\cC_{> k} \coloneqq \bigcup_{j> k}\cC_{j}$. We also define  
\[
L_{k} \coloneqq \sum_{\Gamma \in \cC_{k}} \omega(\Gamma), \qquad L_{\leq k} \coloneqq \sum_{\Gamma \in \cC_{\leq k}} \omega(\Gamma) \qquad \text{and} \qquad L_{> k} \coloneqq \sum_{\Gamma \in \cC_{> k}} \omega(\Gamma).
\]
Then the cluster expansion of a polymer model takes the form $\log \Xi = L_{\le k} + L_{>k}$.

\section{Slow mixing of the Glauber dynamics} \label{sec:slowmixing}

In this section, we prove \Cref{thm:slowmixing}, which shows that the Glauber dynamics for the antiferromagnetic Ising model mixes exponentially slowly for a collection of graphs $G$ in the class $\cH(n, d, \codeg, \ex)$ with constant $\codeg \geq 1, 0 \leq \ex <2$, and with $\lambda, \beta > 0$ satisfying $\lam \le \lam_0$ and $\alpha \coloneqq \lam (1 - e^{-\beta}) \ge \frac{C_0 \log^{3/2}d}{d^{\exx}}$. Assume these conditions from now on.

\subsection{Balanced sets as bottlenecks}
For a collection of configurations $\mathcal{S} \subseteq 2^V$, we define its weight as
\begin{align*}
    \tilde{\omega}(\mathcal{S}) \coloneqq \sum_{S \in \cS} \tilde{\om}(S) = \sum_{S \in \cS} \lambda^{|S|} e^{- \beta |E(S)|}.
\end{align*}
Define
\[\cS_{bal} \coloneqq \{S \subseteq V : |S \cap \cE|=|S \cap \cO|\}
\]
denote the set of balanced vertex sets.
Further, let 
\[
\cS_{\cE} \coloneqq \{S \subseteq V : |S \cap \cE|>|S \cap \cO|\} \qquad \text{ and } \qquad \cS_{\cO} \coloneqq \{S \subseteq V : |S \cap \cE|<|S \cap \cO|\}
\]
denote the majority-even and majority-odd sets, respectively. Without loss of generality, we may assume that $\mu(\cS_{\cE}) \leq 1/2$, where $\mu=\mu_{\lambda, \beta}$ is the Gibbs measure for the Ising model. Our goal is to show that the set $\cS_{\cE}$ has small conductance, through a sequence of lemmas proven later, and to deduce Theorem \ref{thm:slowmixing} from these.

First, we give a bound on the conductance in terms of the weight of balanced vertex sets. For our Markov chain $M$, let $P$ denote the transition matrix of the Glauber dynamics, and recall that $Q(x,y) \coloneqq \pi(x)P(x,y)$. Note that if $S \in \cS_{\cE}$ and $S' \notin \cS_{\cE}$ with $P(S, S') \neq 0$, then $S' \in \cS_{bal}$. Observe that
\begin{align*}
    Q(\cS_{\cE}, \cS_{\cE}^c) = \sum_{\substack{S \in \cS_{\cE}\\ S' \notin \cS_{\cE}}} \mu(S) P(S, S') = \sum_{\substack{S \in \cS_{\cE}\\ S' \notin \cS_{\cE}}} \mu(S') P(S', S)= \sum_{\substack{S \in \cS_{\cE}\\ S' \in \cS_{bal}}} \mu(S') P(S', S) \leq \mu(\cS_{bal}),
\end{align*}
where the second equality uses the reversibility of the chain and the last inequality uses that for any fixed $S' \in \cS_{bal}$ we have $\sum_{S \in \cS_{\cE}} P(S', S) \leq 1$. Thus, the conductance satisfies
\begin{align*}
    \Phi_M \leq \Phi_M(\cS_{\cE}) = \frac{Q(\cS_{\cE}, \cS_{\cE}^c)}{\mu(\cS_{\cE})} \leq \frac{\mu(\cS_{bal})}{\mu(\cS_{\cE})} = \frac{\tilde{\omega}(\cS_{bal})}{\tilde{\omega}(\cS_{\cE})}.
\end{align*}
Note that $\tilde{\omega}(\cS_{\cE}) \geq \sum_{S \subseteq \cE} \lambda^{|S|} \geq (1+\lambda)^{n/2}$, using that $|E(S)|=0$ for any $S \subseteq \cE$. This implies that
\begin{equation} \label{eq:conductance_bound_weight_balanced}
    \Phi_M \leq \frac{\tilde{\omega}(\cS_{bal})}{(1+\lambda)^{n/2}}.
\end{equation}

The remaining goal is to get an upper bound for $\tilde{\omega}(\cS_{bal})$. To do this, we cover the family $\cS_{bal}$ of balanced sets by families of sets which are either small or large on both sides. The transition point from small to large will depend on $f(\lambda)$ as defined in \eqref{eq:def-flambda}. We set
\[
\cS_{small}  \coloneqq \{S \subseteq V : |S \cap \cO|, |S \cap \cE| \leq f(\lam)\frac{n}{2}\}, \]
\[
\cS_{big} \coloneqq \{S \subseteq V : |S \cap \cO|, |S \cap \cE| > f(\lam)\frac{n}{2}\}.
\]
Note that $\cS_{bal} \subseteq \cS_{small} \cup \cS_{big}$.
We further partition the set $\cS_{big}$ as follows: for $\cD \in \{\cO, \cE\}$,
\[
\cS_{big, \cD} \coloneqq \{S \in \cS_{big} : |[S \cap \cD]| \leq 3n/8\},
\]
\[
\cS_{rest} \coloneqq \cS_{big} \setminus (\cS_{big, \cE} \cup \cS_{big, \cO}).
\]
We remark that the set $\cS_{rest}$ is a new feature of the Ising model to analyze, not arising in the hard-core model on independent sets in graphs like in the work of Galvin and Tetali \cite{GaTe2006}. This is because the Ising model allows sampling arbitrary subsets $S \subseteq V$, whereas if we only consider independent sets $S \subseteq V$ then they must satisfy either $|[S \cap \cE]| \le n/4$ or $|[S \cap \cO]| \le n/4$, so that $\cS_{rest}$ is empty in the hard-core model.

The remaining task is to bound the weight of the sets $\cS_{small}, \cS_{big, \cD}, \cS_{rest}$ separately, given by the following three lemmas.

\begin{lemma} \label{lem:small}
We have
\begin{align*}
    \tilde{\omega}(\cS_{small}) \leq (1+\lambda)^{n/4}.
\end{align*}
\end{lemma}

\begin{lemma} \label{lem:middle}
There exists a constant $C'>0$ such that
\begin{align*}
    \tilde{\omega}(\cS_{big, \cD}) \leq (1+\lambda)^{n/2} \exp\left( - \frac{C' \alpha^2 f(\lam)}{d^{\ex} \log d} \cdot n \right).
\end{align*}
\end{lemma}

\begin{lemma} \label{lem:rest}
We have
\begin{align*}
    \tilde{\omega}(\cS_{rest}) \leq (1+\lambda)^{n/2} \exp(-\Omega(\lam n)).
\end{align*}
\end{lemma}

\Cref{lem:small} is proven by direct calculation in \Cref{sec:small}. \Cref{lem:middle} requires a more involved calculation given in \Cref{sec:container}. In fact, it requires slight modifications of Galvin's container construction \cite{galvin2011threshold, galvin2019independent}, which is common in this area, to bounding weights of sets that are not necessarily 2-linked. Some of the details are deferred to \Cref{sec:appendix_containers}. Although it seems folklore that such a result is possible, it was not done explicitly before to our knowledge. The proof of Lemma \ref{lem:rest} is only a slight modification of a result from our previous work \cite[Lemma 7.1]{GeKaSaWdo25}, and we outline the basic changes to the proof in \Cref{sec:non-polymer}.

Assuming the above three lemmas, we now prove \Cref{thm:slowmixing}.
\begin{proof}[Proof of \Cref{thm:slowmixing}]
Since $\cS_{bal} \subseteq \cS_{small} \cup \cS_{big, \cE} \cup \cS_{big, \cO} \cup \cS_{rest}$, we have 
    \[
    \tilde{\omega}(\cS_{bal}) \leq \tilde{\omega}(\cS_{small}) + \tilde{\omega}(\cS_{big, \cE}) + \tilde{\omega}(\cS_{big, \cO}) + \tilde{\omega}(\cS_{rest}).
    \]
    Plugging in \Cref{lem:small}, \Cref{lem:middle} and \Cref{lem:rest} yields that
    \begin{align} \label{eq:bound-sum}
        \tilde{\omega}(\cS_{bal}) \leq &(1+\lambda)^{n/2}\left[(1+\lam)^{-n/2}+2\exp\left( - \frac{C' \alpha^2 f(\lam)}{d^{\ex}\log d} \cdot n\right)+\exp(-\Omega(\lam n))\right]
        \end{align}
    We observe that the second summand is the dominating term.
    This follows from
    \[
    \frac{C' \alpha^2 f(\lam)}{d^{\ex}\log d} \cdot n \leq \frac{C' \lam_0^2}{R} \cdot \frac{1 }{d^{\ex}\log d} \cdot \lam n = o(\lam n) \hspace{10pt} \text{and} \hspace{10pt} \frac{C' \alpha^2 f(\lam)}{d^{\ex}\log d} \cdot n = o\left(\frac{n}{4} \log(1+\lam)\right),
    \]
    where the first upper bound uses $\alpha \leq \lam \leq \lam_0$, and the second bound uses $\log(1+\lambda) \geq \lam - \frac{\lam^2}{2}$. Thus, \eqref{eq:bound-sum} shows that there exists a constant $C > 0$ such that 
    \begin{align} \label{eq:weight_bal}
        \tilde{\omega}(\cS_{bal}) \leq (1+\lambda)^{n/2} \exp\left( - \frac{C \alpha^2 f(\lam) }{2 d^{\ex}\log d} \cdot n\right).
    \end{align}
    Plugging \eqref{eq:weight_bal} into \eqref{eq:conductance_bound_weight_balanced}, we conclude that the conductance of our Markov chain $M$ satisfies $\Phi_M \leq  \exp\left( - \frac{C \alpha^2 f(\lam)}{2 d^{\ex} \log d} \cdot n \right)$. Therefore, by \Cref{lem:conduc_mix}, the mixing time satisfies
    \[
    \tmix(M) \geq \exp\left( \frac{C \alpha^2 f(\lam)}{d^{\ex}\log d} \cdot n \right). \qedhere
    \]
\end{proof}

\subsection{The contribution of small sets -- proof of \Cref{lem:small}} \label{sec:small}

In this subsection, we prove our upper bound on the total weight of $\cS_{small}$, which consists of subsets $S \subseteq V$ with $|S \cap \cO|, |S \cap \cE| \leq f(\lam)\frac{n}{2}$. The bound will follow from elementary considerations. Recall that $f(\lam) = \frac{\lam}{R \log^2(1/\lam)}$ if $0 < \lam < \frac{1}{e}$, and $f(\lam) = \frac{\lam}{R}$ if $\frac{1}{e} \le \lam \le \lam_0$.
\begin{proof}[Proof of \Cref{lem:small}]

We calculate that
\begin{align*}
    \tilde{\omega}(\cS_{small}) &= \sum_{S \in \cS_{small}} \lambda^{|S|} e^{- \beta |E(S)|} \leq \binom{\frac{n}{2}}{\leq f(\lam) \frac{n}{2}}^2 (1+\lambda)^{f(\lam)n} \\
    &\leq \exp\left( n \cdot f(\lam) \log\left(\frac{e}{f(\lam)}\right) \right) \cdot  (1+\lambda)^{f(\lam)n},
\end{align*}
where the first inequality uses $e^{-\beta |E(S)|} \le 1$ and the second one uses \eqref{eqn:binom_estimate}. Since $\lam \le \lam_0$ is bounded, we may choose $R$ large enough so that $f(\lam) \le \frac{1}{8}$. Now it suffices to show that $f(\lam) \log \left( \frac{e}{f(\lam)}\right) \le \frac{1}{8} \log(1+\lam)$. 

If $\lam \ge 1/e$, then we may choose $R$ large enough so that $f(\lam) \log \left( \frac{e}{f(\lam)}\right) \le \frac{1}{8}\log\left(1+\frac{1}{e}\right) \le \frac{1}{8} \log(1+\lam)$. If $\lam < 1/e$, then we have
\begin{align*}
    f(\lam) \log \left( \frac{e}{f(\lam)}\right) = \frac{\lam}{R \log^2(1/\lam)} \log \left( \frac{e R \log^2(1/\lam)}{\lam} \right) \le \frac{2 \lam}{R \log(1/\lam)} \le \frac{2\lam}{R} \le \frac{1}{8} \log(1+\lam),
\end{align*}
where $R$ is chosen sufficiently large so that all the inequalities hold. This finishes the proof.
\end{proof}

\subsection{The contribution of large sets -- container lemma and proof of \Cref{lem:middle}} \label{sec:container}

Fix $\cD \in \{\cE, \cO\}$, and let $\overline{\cD}$ denote the complementary bipartition side. In this subsection, we upper bound the total weight of $\cS_{big, \cD}$, which consists of subsets $S \subseteq V$ satisfying $|S \cap \cO|, |S \cap \cE| > f(\lam) \frac{n}{2}$ and $|[S \cap \cD]| \le 3n/8$.

This will use a container lemma similar to those in \cite{galvin2011threshold, GaTe2006, GeKaSaWdo25, jenssen2020independent}, giving upper bounds on the number of vertex sets with a given closure and neighborhood. Towards this, define, for $a, b \in \mathbb{N}$,
\begin{align*}
\cG(a, b) = \cG_{\cD}(a, b) &\coloneqq \{A \subseteq \cD : |[A]|=a, |N(A)|=b\},\\
\cS(a, b) = \cS_{\cD}(a, b) &\coloneqq \{S \in \cS_{big, \cD} : S \cap \cD \in \cG(a, b)\}.
\end{align*}
Also define the total weight of $\cG(a, b)$ to be
\[
\omega(\cG(a,b)) \coloneqq \sum_{A \in \cG(a, b)} \sum_{B \subseteq N(A)} \frac{\lambda^{|A|+|B|}}{(1+\lambda)^{|N(A)|}} e^{- \beta |E(A, B)|}.
\] 
The following container lemma gives an upper bound on $\om(\cG(a, b))$. It is almost the same as \cite[Lemma 4.2]{GeKaSaWdo25}, but the difference here is that sets in $\cG(a, b)$ need not be $2$-linked. We give a proof of this new container lemma in \Cref{sec:appendix_containers}.

\begin{lemma} \label{lem:container_main}
Fix $\codeg \ge 1$, $0 \le \ex < 2$, and $\lam_0 > 0$, and take $C_0 > 0$ to be sufficiently large. Suppose that $G \in \mathcal{H}(n,d,\codeg,\ex)$, that $\lambda \le \lam_0$ and $\alpha \coloneqq \lam(1 - e^{-\beta}) \ge \frac{C_0 \log^{3/2} d}{d^{\exx}}$, and that $a,b \in \mathbb{N}$ satisfy $b \ge \left(1 + \Omega\left(\frac{1}{d^{\ex}}\right)\right)a$. Then for any $a \geq f(\lam) \frac{n}{2}$, we have
    \begin{align*}
        \omega(\cG(a, b)) \leq \frac{n}{2}\exp \left( -\frac{C(b - a)\alpha^2}{\log d} \right).
    \end{align*}
\end{lemma}

We now use \Cref{lem:container_main} to prove \Cref{lem:middle}.

\begin{proof}[Proof of \Cref{lem:middle}]
Recalling that for every $S \in \cS_{big, \cD}$ we have $a \geq f(\lambda) \frac{n}{2}$, we may split up the weight $\tilde{\omega}(\cS_{big, \cD})$ as follows:
\begin{align*}
    \tilde{\omega}(\cS_{big, \cD}) \leq  \sum_{b \geq a \geq f(\lambda) \frac{n}{2}} \sum_{S \in \cS(a, b)} \lambda^{|S|} e^{- \beta |E(S)|},
\end{align*}
where the first sum is over both $a, b \in \mathbb{N}$.
Since every $S \in \cS(a, b)$ satisfies $S \cap \cD \in \cG(a, b)$, we may rewrite $A=S \cap \cD$ and $B_0= S \cap \overline{\cD}$ and obtain that
\begin{align*}
    \tilde{\omega}(\cS_{big, \cD}) &\leq \sum_{b \geq a \geq f(\lambda) \frac{n}{2}} \sum_{A \in \cG(a, b)} \lambda^{|A|} \sum_{B_0 \subseteq \overline{\cD}} \lambda^{|B_0|} e^{- \beta |E(A, B_0)|} \\
    &\leq \sum_{b \geq a \geq f(\lambda) \frac{n}{2}} \sum_{A \in \cG(a, b)} \sum_{B \subseteq N(A)} \lambda^{|A|+|B|} e^{- \beta |E(A, B)|} \sum_{B' \subseteq \overline{\cD} \setminus N(A)} \lambda^{|B'|}.
\end{align*}
The second inequality splits every $B_0 = S \cap \overline{\cD}$ into $B= B_0 \cap N(A)$ and $B' = B_0 \cap (\overline{D} \setminus N(A))$.
Using that $\sum_{B' \subseteq \overline{\cD} \setminus N(A)} \lambda^{|B'|} \leq (1+\lam)^{n/2}$, we then get that
\begin{align*}
    \tilde{\omega}(\cS_{big, \cD}) \leq & (1+\lambda)^{n/2} \sum_{b\geq a \geq f(\lambda) \frac{n}{2}} \sum_{A \in \cG(a, b)} \sum_{B \subseteq N(A)} \frac{\lambda^{|A|+|B|}}{(1+\lambda)^{|N(A)|}} e^{- \beta |E(A, B)|}\nonumber \\ 
    \leq & (1+\lambda)^{n/2} \left( \frac{n}{2} \right)^2 \max_{b\geq a \geq f(\lam) \frac{n}{2}} \omega(\cG(a, b)),
\end{align*}
where the second inequality uses that there are at most $n/2$ choices for both $a$ and $b$.

We may now apply \Cref{lem:container_main} to finish the proof. Indeed, our assumption on $G$ implies that $b - a \geq a \cdot \Omega\left(\frac{1}{d^{\ex}}\right) $ (otherwise $\cG(a,b)$ is empty), so that \Cref{lem:container_main} applies. Moreover, polynomial factors of $n$ are negligible compared to the resulting exponential upper bound given by \Cref{lem:container_main}.
\end{proof}

\subsection{The contribution of remaining sets -- proof of \Cref{lem:rest}} \label{sec:non-polymer}

We aim to bound the contribution of $\cS_{rest} \coloneqq \cS_{big} \setminus (\cS_{big, \cE} \cup \cS_{big, \cO})$. The proof follows that of \cite[Proof of Lemma 7.1]{GeKaSaWdo25} with a modification in the choice of the parameters. For the current proof, set $m \coloneqq \frac{\lambda n}{\log d}$, so as to make the dependence on $\lambda$ explicit. Also set $s \coloneqq \frac{d}{2} \cdot \frac{\lam}{1+\lam}$ and let
\[
   \cJ_{\cD} \coloneqq\{S \subseteq V : \exists \text{ at least } m \text{ vertices } v \in \cD \text{ such that } 1 \leq |N(v) \cap S| \leq s\},
\]
and
\[
    \cJ' \coloneqq \cS_{rest} \setminus (\cJ_{\cE} \cup \cJ_{\cO}). 
\]
Then $\cS_{rest} \subseteq \cJ_{\cO} \cup \cJ_{\cE} \cup \cJ'$.
Following the calculations as in \cite{GeKaSaWdo25}, one may show that
\[
\tilde{\omega}(\cJ_{\cD}) \leq (1+\lambda)^{n/2} \qquad  \text{ and } \qquad \tilde{\omega}(\cJ')\leq \exp(-\Omega(n \log^2 d)).
\]
\Cref{lem:rest} now easily follows (noting that $\lambda \leq \lambda_0=o(\log^2 d)$).

\section{Efficient sampling via the cluster expansion} \label{sec:sampling}

In this section, we prove \Cref{thm:efficient_sampling}, giving an efficient sampling algorithm for the Ising model with Gibbs measure $\mu_{\lam,\beta}$ for graphs $G$ in the class $\cH'(n,d,\codeg,\ex)$, as well as an FPTAS for computing their partition functions $Z_G(\lam,\beta)$. We use the method of cluster expansion, and we assume that $d$ is sufficiently large but independent of $n$ throughout this section.

In their work on counting independent sets in the hypercube $Q^d$, Jenssen and Perkins \cite{jenssen2020independent} used an appropriate polymer model and the convergence of the cluster expansion to produce an approximation to the partition function in the hard-core model. Importantly, the method inherently uses an auxiliary measure that approximates the desired sampling measure. In principle, this can produce a sampling algorithm if the steps of the proof of the approximation can be translated from computation to efficient sampling. This is the idea of \cite{JePePo23} applied to the hard-core model. In the case of our Ising model, the cluster expansion has been proven in \cite{GeKaSaWdo25} to be absolutely convergent for graphs in the class $\cH'(n,d,\codeg,\ex)$ and ranges of $\lam,\beta$ that we consider through the paper, via a verification of the Koteck\'y--Preiss condition \cite{kotecky1986cluster}. Thus, it is expected that the sampler can be reproduced accordingly. In this section, we verify the efficiency of this sampler for the Ising model, in a condensed form.

\subsection{The polymer model}
 
We use the polymer models described in \cite{GeKaSaWdo25, kronenberg2022independent}. For $\cD \in \{\cO, \cE\}$, the weight of a $2$-linked set $A \subseteq \cD$ is defined as
$$\om(A)\coloneqq\sum_{B\subseteq N(A)}\frac{\lam^{|A|+|B|}}{(1+\lam)^{|N(A)|}}e^{-\beta|E(A,B)|} = \sum_{B \subseteq N(A)} \om(A, B),$$
where
$$\om(A,B)\coloneqq\frac{\lam^{|A|+|B|}}{(1+\lam)^{|N(A)|}}e^{-\beta|E(A,B)|}.$$
We consider the polymer models
\begin{align*}
    \cP_{\cD}= \left\{ A \subseteq \cD: \quad A \text{ is $2$-linked},\quad |[A]|\leq \frac{3}{4}\cdot |\cD|\right\},
\end{align*}
where two polymers $A,A' \in\cP_\cD$ are \textit{compatible} if $A\cup A'$ is $2$-linked, and the weight function is given by $\om(A)$. Further, we define the \textit{decorated} polymer models
\begin{align*}
\hat{\cP}_\cD \coloneqq \left\{(A,B): \quad A\subseteq \cD\ \text{is $2$-linked},\quad |[A]|\leq \frac{3}{4}\cdot |\cD|, \quad \text{and} \quad B \subseteq N(A)\right\},
\end{align*}
where the elements $(A,B)\in\hat\cP_\cD$ are the \textit{decorated} polymers. Two decorated polymers $(A,B)$, $(A',B')$ are compatible if $A\cup A'$ is $2$-linked, and the weight function is given by $\om(A,B)$. The use of the decorated polymer model is to decompose the weight $\om(A)$ of a polymer $A \in \cP_{\cD}$ to its summands $\om(A,B)$. 

A set of pairwise compatible decorated polymers with respect to $\cD$ is called a \textit{decorated polymer configuration}, and the collection of such polymers is denoted by $\hat{\Om}_{\cD}$. The weight of a decorated polymer configuration $\hat{\Theta}$ from $\hat{\Om}_{\cD}$ is defined as
\begin{align*}
    \omega(\hat \Theta) \coloneqq \prod_{(A, B) \in \hat{\Theta}} \omega(A, B).
\end{align*}
Similarly, the collection of polymer configurations from $\cP_{\cD}$ is denoted by $\Om_{\cD}$.
The partition functions of these polymers models coincide and are given by
$$\Xi_\cD \coloneqq \Xi_{\hat{\Om}_\cD} = \Xi_{\Om_{\cD}} = \sum_{\hat \Theta \in \hat \Om_\cD}\omega(\hat{\Theta}).$$
We denote by $\cC_{\cD}$ the collection of clusters with polymers from $\cP_{\cD}$, and we denote by $L_{\cD, k}$, $L_{\cD, \leq k}$, $L_{\cD, >k}$ the corresponding terms and truncations of its cluster expansion.

We define a new measure $\hat{\mu}^*$ on pairs $(S, \cD)$ using the following steps:
\begin{enumerate}
    \item Choose a \textit{defect side} $\cD\in\{\cO,\cE\}$ with probability proportional to $\Xi_{\cD}$.
    \item Sample a decorated polymer configuration $\hat{\Theta}$ from $\hat{\Om}_{\cD}$ according to the measure $\nu_\cD$ defined by
    $$\mathbb{P}_{\nu_{\cD}}(\boldsymbol{\hat{\Theta}}=\hat{\Theta})=\frac{\om(\hat{\Theta})}{\Xi_\cD}.$$
    \item If $\hat{\Theta}=\{(A_1,B_1),\dots, (A_s,B_s)\}$, then let $D \coloneqq \bigcup_i(A_i\cup B_i)$, and put 
    \begin{itemize}
        \item every vertex of $D$, and
        \item every vertex $v\in\overline{\cD} \setminus N(D)$ independently with probability $\frac{\lam}{1+\lam}$
    \end{itemize}
    into vertex set $S$.
\end{enumerate}
Then $\hat{\mu}^*$ is a measure on the set of pairs $(S, \cD)$ where $S \subseteq V$ is a vertex subset and $\cD \in \{\cO, \cE\}$ is a defect side. Given $(S, \cD)$, we may recover the decorated polymer configuration $\hat{\Theta}(S) \in \hat{\Om}_{\cD}$ by taking the maximal $2$-linked components of $\cD$ to form the $A_i$ and the vertices in $S \cap N(A_i)$ to form the $B_i$. Using this, we also define
\begin{align*}
    \hat{\om}^*(S, \cD) \coloneqq \mathbbm{1}_{\hat{\Theta}(S) \in \Om_\cD} \lam^{|S|} e^{- \beta |E(S)|},
\end{align*}
and
\begin{align*}
\hat{\om}(S) \coloneqq \hat{\om}^*(S, \cO)+\hat{\om}^*(S, \cE).
\end{align*}
Now we define a measure $\hat{\mu}$ by sampling subsets $S \subseteq V$ proportional to their weights $\hat{\om}(S)$:
\begin{align*}
\mathbb{P}_{\hat{\mu}}(\mathbf{S}=S)=\frac{\hat{\om}(S)}{\hat{Z}},
\end{align*}
where $\hat{Z} \coloneqq \sum_{S \subseteq V} \hat{\om}(S)$.
Following \cite{GeKaSaWdo25}, the partition function $\hat{Z}$ of $\hat{\mu}$ satisfies 
$$\hat{Z}=(1+\lam)^{n/2}(\Xi_{\cO}+\Xi_{\cE}).$$
The following lemma states that $\hat{\mu}$ is indeed a good approximation for the Gibbs measure $\mu= \mu_{\lambda, \beta}$ of the Ising model.

\begin{lemma}[Lemma 3.3 in \cite{GeKaSaWdo25}] \label{lem:approximateZ}
Fix $\codeg \ge 1$, $0 \le \ex < 2$, and $\lam_0>0$. Then there exists a constant $C_0 > 0$ such that whenever $\lam \le \lam_0$ and $\lam(1 - e^{-\beta}) \ge \frac{C_0 \log^{3/2} d}{d^{\exx}}$ and $G \in \mathcal{H}(n,d,\codeg,\ex)$, we have
    \[
    \left| \log Z_G(\lam, \beta) - \log\left((1+ \lam)^{n/2} (\Xi_\cO + \Xi_\cE)\right) \right| = O\left( \exp\left(-\frac{n}{d^{\ex+4}}\right)\right),
    \]
and this implies that
\[
\norm{\hat{\mu} - \mu}_{TV} = O\left( \exp\left(-\frac{n}{d^{\ex+4}}\right)\right).
\]
\end{lemma}

In other words, $\hat\mu$ is a good approximation of $\mu$. Thus, it remains to sample (approximately or exactly) from $\hat\mu$. Looking at the steps leading to the definition of $\hat\mu$, we need to sample a defect side with probability $\Xi_\cD/(\Xi_\cE+\Xi_\cO)$ and then sample a decorated polymer model according to the measure $\nu_\cD$. Both of these tasks require computing the polymer partition functions, which cannot be done exactly. Hence, we are aiming for efficient estimates and approximate samplers for all the aforementioned measures. The convergence of the cluster expansion, which is the heart of the proof of \Cref{lem:approximateZ}, is the key to provide such estimates via truncation. The computation of terms of the expansion becomes a subproblem of the algorithm, which has to be done in polynomial time. 

We are now ready to write down a concrete algorithm.

\begin{algorithm}[H]
    \caption{Approximate Sampler for the Ising model\label{alg:main}}
    \begin{enumerate}
        \item Let $\epsilon_0 \coloneqq \exp\left(- \frac{n}{d^{\ex+4} \log d} \right)$. If $\epsilon\leq \epsilon_0$, sample a set via brute force. Otherwise do the following.
        \item Let $k_0$ be chosen according to \Cref{lem:Kotecky}. Compute $L_{\cD,\leq k_0}$ for $\cD\in\{\cE,\cO\}$ by \Cref{lem:compute_L}.
        \item Choose a defect side $\cD$ with probability $\dfrac{L_{\cD,\leq k_0}}{L_{\cE,\leq k_0}+L_{\cO,\leq k_0}}$.
        \item Sample a decorated polymer configuration $\hat{\Theta}$ from $\nu_\cD$ approximately. This is done in \Cref{lem:compute_nu}. Let $D$ be the set of all vertices in $\hat\Theta$.
        \item Let $S'$ be a random set in which each vertex in $\overline{\cD}\setminus D$ is contained with probability $\frac{\lam}{1+\lam}$.
        \item Output $S:=D\cup S'$.
    \end{enumerate}
\end{algorithm}
We now state the three lemmas that are called upon in \Cref{alg:main}.

\begin{lemma} \label{lem:compute_L}
    Given $\cD \in \{\cO, \cE\}$, there is an algorithm to compute $L_{\cD, \leq k}$ in time 
    \[
    O(n^2 k^8) \exp(O(kd)).
    \]
\end{lemma}

\begin{lemma} \label{lem:compute_nu}
    Given $\cD \in \{\cO, \cE\}$ and $\epsilon \geq \exp\left( - \frac{n}{d^{\ex+4} \log d}\right)$, there is a sampling scheme that samples $\hat{\Theta}$ from $\hat{\Omega}_{\cD}$ according to distribution $\hat{\nu}_{\cD}$ satisfying 
    \[
    \lVert \hat{\nu}_{\cD} - \nu_{\cD} \rVert_{TV} \leq \epsilon \quad \text{ in time } \quad \left( \frac{n}{\epsilon}\right)^{O(d^4)}.
    \]
\end{lemma}

\begin{lemma}\label{lem:Kotecky}
    Fix $\epsilon>0$, $\codeg \geq 1$, and $0 \le \ex < 2$. If $\lam \le \lam_0$ and $\lam(1 - e^{-\beta}) \ge \frac{C_0 \log^{3/2}d}{d^{\exx}}$ and $G\in\cH'(n,d,\codeg,\ex)$, then there exists $k_0\in \{\lfloor d^3\log n\rfloor, \lfloor d^3\log n+1\rfloor, \lceil d^{\ex+1}\log(16n/d^{\ex+1}\epsilon)\rceil\}$ such that
    $$|\Xi_\cD-L_{\cD,< k_0}|\leq \epsilon/16.$$
\end{lemma}

Lemmas \ref{lem:compute_L} and \ref{lem:compute_nu} are proven in Sections \ref{sec:sample_small_cluster} and \ref{sec:compute_nu}, respectively. \Cref{lem:Kotecky}, called upon in step (2) of \Cref{alg:main}, demonstrates that clusters of size at most $k_0$ are enough to compute the partition function closely enough. The statement of the lemma is implied by the tail estimates for the cluster expansion described in \cite{GeKaSaWdo25}, with slight modifications. The proof is deferred to \Cref{sec:appendix_tails}, which also explains what the right choice of $k_0$ is.

Assuming the above three lemmas, we may now give the proof of \Cref{thm:efficient_sampling}.

\begin{proof}[Proof of \Cref{thm:efficient_sampling}]
    Start by setting $\epsilon_0 \coloneqq \exp\left(- \frac{n}{d^{\ex+4} \log d} \right)$.

    If $\epsilon \leq \epsilon_0$, then we sample subset $S \subseteq V$according to the Gibbs measure $\mu=\mu_{\lambda, \beta}$ in brute-force fashion. There are at most $2^n$ subsets of $V$ and evaluating their weight is polynomial in $n$. Thus, we may sample a set according to the probability $\mu$ in time
    $2^{n+o(n)}\le \left( \frac{1}{\epsilon}\right)^{d^6 \log d},$
    since $\ex<2$.

    Thus, we may assume that $\epsilon \geq \epsilon_0$. By \Cref{lem:approximateZ}, we have
    \begin{align*}
        \lVert\hat \mu - \mu \rVert_{TV} = O\left( \exp\left( - \frac{n}{d^{\ex+4}}\right)\right) \leq \frac{\epsilon}{2}.
    \end{align*}
    This means that it is enough to sample from $\hat{\mu}$ with an error of at most $\epsilon/2$.
    Hence, we need to pick a defect side $\cD$ with the correct distribution, and then we need to sample from $\hat{\mu}^*$ given this defect side $\cD$.
    
    The probability that $\cD \in \{\cO, \cE\}$ is indeed the defect side is given by
    \begin{align} \label{eq:prob_defect}
        \mathbb{P}_{\hat \mu^*}\left(\bm{\cD}=\cD \right) = \frac{\Xi_{\cD}}{\Xi_{\cO}+\Xi_{\cE}},
    \end{align}
    and thus we aim to approximate $\Xi_{\cO}$ and $ \Xi_{\cE}$.
    By \Cref{lem:Kotecky}, 
    \begin{align} \label{eq:approx_good}
        |\Xi_{\cD} - L_{\cD, < k_0}| \leq \frac{\epsilon}{16},
    \end{align}
    for the appropriate choice of $k_0$. By \Cref{lem:compute_L}, we may compute $L_{\cD, \leq k_0}$ in time 
    \begin{align} \label{eq:time_forL}
        O(n^2 k_0^8) \exp(O(k_0 d)).
    \end{align}
    If $k_0=\lfloor d^3\log n \rfloor$ or $k_0=\lfloor d^3\log n +1\rfloor$, then the time in \eqref{eq:time_forL} is $ O(n^{2}d^{24}\log^{8}n)n^{O(d^4)}$. If $k_0=\lceil d^{\ex+1}\log(16n/d^{\ex+1}\epsilon)\rceil$, then the time in \eqref{eq:time_forL} is $O\left(n^{2}d^{8(\ex+1)}\log^8\left(\frac{n}{\epsilon}\right)\right)\left(\frac{n}{\epsilon}\right)^{O(d^{\ex+2})}$. 
    
    In any case, since $\ex<2$ the running time for computing $L_{\cD,\leq k_{0}}$ is at most
    \begin{equation*}\label{eqn:cluster_time}
        \left(\frac{n}{\epsilon}\right)^{O(d^4)}.
    \end{equation*}  
    Now pick the defect side $\bm{\cD}$ to be $\cD$ with probability
    \begin{align*}
    \frac{L_{\cD, \leq k_0}}{L_{\cO, \leq k_0} + L_{\cE, \leq k_0}}.
    \end{align*}
    Using \eqref{eq:approx_good} and noting that $\Xi_{\cO}, \Xi_{\cE} \geq 1$, this probability approximates \eqref{eq:prob_defect} up to an error of at most $\epsilon/8$.
    Given a defect side $\cD$, by \Cref{lem:compute_nu} we may also sample a decorated polymer configuration $\hat{\Theta} \in \hat{\Omega}_{\cD}$ in time polynomial in $n$ and $\frac{1}{\epsilon}$ with error $\frac{\epsilon}{8}$.
    Construct a subset $S \subseteq V$ containing each vertex in $\hat{\Theta}$. Then add every vertex $v\in\overline{\cD} \setminus N(D)$ independently with probability $\frac{\lam}{1+\lam}$ to $S$, recalling that $D \coloneqq \bigcup_{(A,B) \in \hat{\Theta}}(A\cup B)$. Output $S$. This gives our polynomial-time sampling scheme.
    
    Finally, the above proof also gives an FPTAS for computing the partition function $Z_G(\lam,\beta)$: In the case $\epsilon\leq\epsilon_0$, we compute the partition function directly, while for $\epsilon\geq\epsilon_0$ we gave an approximation by approximating the partition functions of the polymer models. Indeed,
    \begin{align*}
        \left|1-\frac{Z}{(1+\lam)^{n/2}(L_{\cE,\leq{k_0}}+L_{\cO,\leq{k_0}})}\right| &=\left|1-\frac{Z}{(1+\lam)^{n/2}(\Xi_{\cE}+\Xi_{\cO})}\cdot\frac{\Xi_{\cE}+\Xi_{\cO}}{L_{\cE,\leq{k_0}}+L_{\cO,\leq{k_0}}}\right|\\
        &\leq |1-e^{\epsilon/2}(1-\epsilon/8)|\leq\epsilon,
    \end{align*}
    where the first inequality uses \Cref{lem:approximateZ} and equation \eqref{eq:approx_good}. This finishes the proof.
\end{proof}

\subsection{Sampling clusters of bounded size -- proof of \Cref{lem:compute_L}} \label{sec:sample_small_cluster}

In this subsection, we aim to efficiently compute a truncation of the cluster expansion, $L_{\cD, \leq k}$. The computations performed will be exact: we will list all polymers and clusters of the appropriate sizes and compute their weights exactly to derive $L_{\cD, \leq k}$. We follow an approach from \cite[Theorem 8]{JeKePe20}.

Recall that a cluster $\Gamma$ is an ordered tuple of polymers whose incompatibility graph is connected, and whose size is $\norm{\Gamma} \coloneqq \sum_{(A, B) \in \Gamma} |A|$. We denote $\mathcal{C}_{\cD, \leq k} \coloneqq \{\Gamma\in \cC_\cD: \norm{\Gamma} \leq k\}$ and
\[
L_{\cD, \leq k} \coloneqq \sum_{\Gamma \in \cC_{\cD, \leq k}} \omega(\Gamma).
\]
Also recall that a decorated polymer $(A, B) \in \hat{\cP}_{\cD}$ consists of a $2$-linked set $A \subseteq \cD$ and a set $B \subseteq N(A)$. Since the compatibility relation only depends on $A$, we first focus on the set $A$ and later sum over all subsets $B \subseteq N(A)$. 

In our non-decorated polymer model $\cP_{\cD}$, a polymer $A \in \cP_{\cD}$ is a $2$-linked set $A \subseteq \cD$ with $|[A]| \leq \frac{3}{4} \cdot |\cD|$, and whose size is given by $|A|$. First, we seek to list all polymers and clusters. Denote $\cP_{\cD, \leq k} \coloneqq \{A \in \cP_{\cD}: \; |A|\leq k\}$. By \Cref{l:counting2linked}, there are at most $n(ed^2)^k = n \exp(O(k \log d))$ many $2$-linked sets of size $k$ in our $n$-vertex, $d$-regular graph $G$. Therefore, there are at most
$n\cdot k \exp(O(k \log d))$ polymers in $\cP_{\cD, \leq k}$. Furthermore, by \cite[Lemma 3.4]{PaRe17} we can list all of them in time 
\begin{align} \label{eq:list_polymers}
    O(n^2 k^7 (ed^2)^{2k})= O(n^2 k^7 \exp(4k \log d)).
\end{align}
Using the algorithm in \cite[Proof of Theorem 6]{HePeRe20}, when given a list of the polymers in $\cP_{\cD, \leq k}$, we may list all clusters $\Gamma$ of $\cP_{\cD}$ with $\lVert \Gamma \rVert \leq k$ in time
\begin{align} \label{eq:list_clusters}
    n^2 \exp\left(O(k \log d)\right).
\end{align}
Now, we seek to compute the weights of these polymers and clusters. Given a non-decorated polymer $A$, recall that its weight is
\[
\omega(A) \coloneqq \sum_{B \subseteq N(A)} \frac{\lambda^{|A|+|B|}}{(1+\lambda)^{|N(A)|}} e^{-\beta |E(A, B)|}.
\]
If $|A| \le k$, then there are at most $2^{kd}$ subsets of $N(A)$, and given such a subset the number of edges $|E(A, B)|$ is determined. Thus, we can compute the weight of a polymer $A$ of size at most $k$ in time $\exp(O(kd))$.

To compute the weight of a cluster, we need the weight of the polymers that it consists of and the Ursell function $\phi$ of its incompatibility graph. Given two polymers $A_1, A_2 \in \cP_{\cD, \leq k}$, checking their incompatibility is equivalent to checking whether $A_1 \cup A_2$ is connected in $G^2$. Since $|A_1 \cup A_2| \leq 2k$ and the degree in $G^2$ is at most $d^2$, using breadth-first search the incompatibility graph of a cluster in $\cC_{\cD, \leq k}$ can be constructed in time $O(k^3d^2)$.
By \cite[Lemma 5]{HePeRe20} (see also \cite{BjKoHuKa08} for a more general result on the Tutte polynomial), the Ursell function of the incompatibility graph can be computed in time $\exp(O(k))$. Combining this with the time $\exp(O(kd))$ it takes to compute the weight of a single polymer, we may thus compute the weight of a cluster $\Gamma$ with $\lVert \Gamma \rVert \leq k$ in time at most
\begin{align} \label{eq:weight_cluster}
    k \exp(O(kd)),
\end{align}
since $\Gamma$ contains at most $k$ polymers from $\cP_{\cD, \leq k}$.

Combining \eqref{eq:list_polymers}, \eqref{eq:list_clusters}, and \eqref{eq:weight_cluster}, we get an algorithm that runs in time
\begin{align*}
    \left( n^2 k^7 \exp(4k \log d) + n^2 \exp(O(k \log d))\right) \cdot k \exp(O(kd)) = n^2 k^8 \exp(O(kd)),
\end{align*}
which finishes the proof of Lemma \ref{lem:compute_L}.

\subsection{Sampling a polymer configuration -- proof of \Cref{lem:compute_nu}} \label{sec:compute_nu}

In this subsection, given a defect side $\cD$, we aim to sample a polymer configuration $\hat{\Theta}$ from $\hat{\Omega}_{\cD}$ according to the distribution $\nu_{\cD}$.

Let $\epsilon \geq \epsilon_0= \exp \left( - \frac{n}{d^{\ex+4} \log d}\right)$. By \Cref{lem:Kotecky}, we get some $k_0 \ge 1$ such that
\begin{equation}\label{eq:error_truncation}
    |\Xi_{\cD} - L_{\cD, \leq k_0}| \leq \frac{\epsilon}{16}.
\end{equation}
To sample a polymer configuration, we follow a procedure from \cite{HePeRe20} (also used in \cite{JeKePe20}).
Given a polymer configuration $\hat \Theta \in \hat \Omega_{\cD}$ and a set of vertices $S$, we denote
$$
\hat \cP(\hat \Theta, S) \coloneqq \{ (A, B) \in \hat \cP_{\cD} : A \cap S = \varnothing, (A, B) \text{ compatible with } \hat \Theta\}.
$$
This is the set of polymers $(A, B)$ such that $A$ does not intersect $S$ and $\hat \Theta \cup \{ (A, B) \}$ is still a decorated polymer configuration.
Recall that for any set of polymers $\cP$ with the set of polymer configurations denoted by $\Omega(\cP)$, its polymer model partition function $Z(\cP)$ is given by
\[
Z(\cP) \coloneqq \sum_{\Theta \in \Omega(\cP)} \prod_{(A, B) \in \Theta} \omega(A, B).
\]
Using this, we may sample a polymer configuration $\hat{\Theta}$ from $\nu_{\cD}$ via the following procedure.

\begin{algorithm}[H] 
    \caption{Sampler for $\nu_{\cD}$ \label{alg:sample-nu}}
    \begin{enumerate}
        \item Fix an ordering $v_1, ..., v_n$ of the vertices of $G$. Set $\hat \Theta_0 \coloneqq \varnothing$ and $S_0 \coloneqq \varnothing$.
        \item For $j=0,1,\ldots,n-1$:
        \begin{enumerate}
            \item Collect  all decorated polymers $(A, B)$ from $\hat \cP( \hat \Theta_j, S_j)$ such that $v_{j+1} \in A$ and put them into a list $\cL(j)$.
            \item Sample a polymer from $\cL(j) \cup \{\varnothing\}$ with probability
            \begin{equation} \label{eq:probability_polymer}
                \Pr{\mathbf{(A, B)}= (A, B)} \coloneqq \frac{\omega(A, B) \cdot Z\bigl(\hat \cP( \hat \Theta_j \cup (A, B), S_j \cup \{v_{j+1}\})\bigr)}{Z\bigl(\hat \cP(\hat \Theta_j, S_j)\bigr)}.
            \end{equation}
            \item Set $\hat \Theta_{j+1} \coloneqq \hat \Theta_j \cup \{(A, B)\}$ and $S_{j+1} \coloneqq S_j \cup \{v_{j+1}\}$.
        \end{enumerate}
        \item Output $\hat \Theta=\hat \Theta_n$.
    \end{enumerate}
\end{algorithm}

By \cite[Lemma 11]{HePeRe20}, the distribution of the output $\hat{\Theta}$ is exactly $\nu_{\cD}$. However, to achieve polynomial time, we need to modify \Cref{alg:sample-nu} and settle for a sufficiently good approximation of $\nu_{\cD}$. 

Fix $\epsilon > 0$. By \eqref{eq:error_truncation}, a polymer configuration sampled from $\hat{\Omega}_{\cD}$ with distribution $\nu_{\cD}$ is contained in $\hat{\Omega}_{\cD, \leq k_0}$ with probability at least $1- \frac{\epsilon}{16}$. Thus, it suffices to consider polymers from $\hat \cP_{\cD, \leq k_0}$. Hence, we modify step (a) as follows: When constructing the list $\cL(j)$, consider only polymers that have size at most $k_0$. Denote this modified list by $\cL_{\leq k_0} (j)$.

To compute the list $\cL_{\leq k_0}(j)$ of polymers $(A, B) \in \hat \cP_{\cD, \leq k_0}(\hat \Theta_j, S_j)$ such that $v_{j+1} \in A$, we use the procedure described in \Cref{sec:sample_small_cluster} to compute all polymers in $\hat \cP_{\cD, \leq k_0}$ and then eliminate polymers that are incompatible with $\hat{\Theta}_j$, where $A \cap S_j \neq \varnothing$ or $v_{j+1} \notin A$, from this list. That is, by \eqref{eq:list_polymers} we can list all polymers in $\hat \cP_{\cD, \leq k_0}$ in time $O(n^2 k_0^7 \exp(4k_0 \log d))$.
Given a polymer from $(A, B) \in \hat \cP_{\cD, \leq k_0}$, it remains to check that $A \cap S_j=\varnothing$, that $v_{j+1} \in A$, and that $(A, B)$ is compatible with $\hat{\Theta}_j$. The first two conditions may easily be checked by going through the at most $k_0$ vertices in $A$. For the third condition, we go through the vertices $v \in \bigcup_{(C, D) \in \hat \Theta_j} C$ and check whether $\text{dist}(v, A) \leq 2$. All in all, checking whether a polymer $(A, B)$ from $\hat \cP_{\cD, \leq k_0}$ is part of $\cL_{\leq k_0}(j)$ takes time $O(n^2)$.
Summarizing, step (a) takes time $O(n^4 k_0^7 \exp(4 k_0 \log d))$.

Now we proceed to step (b). The probability in \eqref{eq:probability_polymer} represents the probability that $(A, B)$ is added to $\hat \Theta$ given that $\hat \Theta_j \subseteq \hat \Theta$. To approximate this probability, we use a fundamental identity from \cite{ScoSo05} (see also \cite{HePeRe20}) which states that for any vertex $v$,
\begin{equation} \label{eq:fundamental_identity}
    Z\bigl(\hat \cP(\hat \Theta, S)\bigr)= Z\bigl(\hat \cP(\hat \Theta, S \cup \{v\})\bigr) + \sum_{\substack{(A, B) \in \hat \cP(\hat \Theta, S)\\ v \in A}} \omega(A, B) \cdot Z\bigl(\hat \cP(\hat \Theta \cup (A, B), S \cup \{v\})\bigr).
\end{equation}
The first summand corresponds to the empty polymer.
Setting 
\[Y(A, B)\coloneqq \omega(A, B) \cdot Z\bigl(\hat\cP(\hat \Theta_j \cup (A, B), S_j \cup \{v_{j+1}\})\bigr)\] and using \eqref{eq:fundamental_identity} we may rewrite the probability in \eqref{eq:probability_polymer} as
\[
\Pr{\mathbf{(A, B)}= (A, B)} = \frac{Y(A, B)}{Y(\varnothing) + \sum_{\substack{(C, D) \in \hat\cP(\hat\Theta_j, S_j)\\ v_{j+1} \in C}} Y(C, D)}.
\]
To approximate this probability, we compute fine approximations to the terms $Y(A, B)$ in this equation.
Set $\epsilon' \coloneqq \frac{\epsilon^2}{160 n^2}$. Given $\epsilon'$-approximations to $Y(A, B)$ for all $(A, B) \in \hat{\cP}(\hat{\Theta}_j, S_j)$, we obtain an $\frac{\epsilon}{8}$-approximation to \eqref{eq:probability_polymer} by \cite[Lemma 5.3]{HePeRe20}.
To get such an $\epsilon'$-approximation to $Y(A, B)$, we need to compute the weight $\omega(A, B)$ as well as a good approximations to partition functions of the form $Z\bigl(\hat{\cP}(\hat{\Theta}_j, S_j)\bigr)$ for $\hat{\Theta}_j \in \hat{\Omega}_{\cD}$ and $S_j \subseteq V$.

All the polymers under consideration are from $\hat{\cP}_{\cD, \leq k_0}$ and their weight can be computed in time $O(k_0 d)$. For the computation of an approximation to the partition function $Z\bigl(\hat{\cP}(\hat{\Theta}_j, S_j)\bigr)$, observe that the polymer models $\hat{\cP}(\hat{\Theta}_j, S_j)$ under consideration are subsets of the original polymer model $\hat{\cP}_{\cD}$. However, all results from \Cref{sec:sample_small_cluster} hold for subsets of $\hat{\cP}_{\cD}$. That is, we may efficiently approximate each partition function $Z\bigl(\hat{\cP}(\hat{\Theta}_j, S_j)\bigr)$ via a truncation of its cluster expansion, as in \Cref{lem:compute_L}. For a formal treatment of this argument, see \cite[Lemma 2.8]{HePeRe20}. Note that we ask for $\epsilon'$-approximations, and by \Cref{lem:Kotecky} one may find $k_0' \in \{\lfloor d^3 \log n \rfloor, \lfloor d^3 \log n+1 \rfloor, \lceil d^{\ex +1} \log \left(\frac{160n^3}{\epsilon^2 d^{\ex+1}}\right) \rceil\}$ such that
$Z\bigl(\hat{\cP}(\hat{\Theta}_j, S_j)\bigr)$ is approximated with error at most $\epsilon'$ by the first $k_0'$ terms of its cluster expansion. This truncation after the first $k_0'$ terms may be computed in time $O(n^2 k_0'^8 \exp(O(k_0' d))$ as in \Cref{lem:compute_L}.

Putting this together, we may compute $Y(A, B)$ for a single $(A, B) \in \hat{\cP}(\hat{\Theta}_j, S_j)$ in time $O(n^2 k_0'^8 \exp(O(k_0' d))$, and there are at most $O(n^2 k_0^7 \exp(4k_0 \log d))$ many such polymers since $(A, B) \in \hat{\cP}(\hat{\Theta}_j, S_j) \subseteq \hat{\cP}_{\cD, \leq k_0}$. Thus, step (b) can be completed in time
$$O(n^4 k_0^7 k_0'^8 \exp(O(k_0' d)).$$
Since the loop in step (2) runs for all $n$ vertices, the runtime of \Cref{alg:sample-nu} with the modifications for step (a) and (b) is
\[
O\bigl(n^5 k_0^7 k_0'^8 \exp\left(O(k_0' d)\right)\bigr)= \left(\frac{n}{\epsilon} \right)^{O(d^{\ex+2})},
\]
where the equality stems from plugging in the possible values of $k_0, k_0'$ as in the proof of \Cref{thm:efficient_sampling}.
Recalling that $\ex<2$ finishes the proof of \Cref{lem:compute_nu}.

\section{Discussion and Open Problems} \label{sec:discussion}

In this paper, we have excluded a natural sampling algorithm, namely the Glauber dynamics, as an efficient sampler for the antiferromagnetic Ising model on a large class of bipartite, regular expander graphs. This extends results of Galvin and Tetali \cite{GaTe2006} on the hard-core model. There are some questions naturally arising from this work. First, we conjecture that the container lemmas we use can be extended up to $\lam(1-e^{-\beta})\geq\tilde\Om(1/d)$, as it is conjectured for the hard-core model in \cite{galvin2011threshold}. From that, the slow mixing result would extend for parameters in this range. 
\begin{question}\label{quest:threshold}
    Is it true that the Glauber dynamics for the antiferromagnetic Ising model mix exponentially slowly for $G\in\cH(n,d,\codeg,\ex)$ as long as $\lam(1-e^{-\beta})=\tilde\Om(1/d)$? 
\end{question}
This would be optimal up to polylog factors, since the Glauber dynamics mix fast below the uniqueness threshold for any antiferromagnetic $2$-spin system \cite{chen2023rapid, li2013correlation}. This threshold is $\sim e/d$ in the hard-core model. We note that in our version of the Ising model, a positive answer to \Cref{quest:threshold} would also be an optimal bound: for our choice of parameters, the uniqueness equation has solution $\hat{x}_d\leq\lam$ which implies we are in the uniqueness regime when $d\lam(1-e^{-\beta})<1$ (see \cite[Definition 2.3]{li2013correlation}), and hence the Glauber dynamics mix rapidly.

On the other hand, equation \eqref{eq:conductance_bound_weight_balanced} tells us that it is enough to bound the contribution of balanced sets $\tilde\om(\cS_{bal}$), while the proof itself uses the container lemma to bound the weight of a much larger family. As this step might be wasteful, it is reasonable to ask whether one can directly show a better bound for balanced sets to answer \Cref{quest:threshold}, bypassing or modifying the container lemma accordingly.
\begin{question}\label{quest:balanced}
    Let $G\in\cH(n,d,\codeg,\ex)$. Is it true that for $\alpha:=\lam(1-e^{-\beta})=\tilde\Om(1/d)$,
    $$\sum_{\substack{S\subseteq V(G)\\ |S\cap\cE|=|S\cap\cO|}}\lam^{|S|}e^{-\beta|S|}\leq (1+\lam)^{n/2}\exp\left(-n\frac{\alpha^2}{\textrm{poly}(d)}\right)?$$
    Specializing to the hard-core model, is it true that for $\lam=\tilde\Om(1/d)$,
    $$\sum_{\substack{I\in\cI(G)\\ |I \cap \cE|=|I \cap \cO|}}\lam^{|I|}\leq (1+\lam)^{n/2}\exp\left(-n\frac{\lam^2}{\textrm{poly}(d)}\right)?$$
\end{question}
We note that the squared term in the exponential is only written here in analogy of previous results and the actual result could be different, as long as it decays exponentially fast with $n$. The problem of counting balanced independent sets in the hypercube $Q^d$ has already been considered in the past, with Park \cite{park2022note} proving that the number of balanced independent sets is 
$$\text{bis}(Q^d)=2^{\left(1-\Theta\left(\frac{1}{\sqrt{d}}\right)\right)2^{d-1}},$$
consistent with \Cref{quest:balanced} for $\lam=1$.

From now on, we focus on the hard-core model for ease of presentation, but the interested reader can reformulate everything for the antiferromagnetic Ising model. Crucially, the local nature of the Glauber dynamics creates a bottleneck between majority-even and majority-odd sets. See also \cite[Section 14]{jenssen2023homomorphisms} for the same phenomenon on the torus in a more general framework. The question arises whether this is \emph{the only} bottleneck, i.e., the only obstacle for rapid mixing. It has been proposed in \cite{HePeRe20, JeKePe20} that a polynomial-time approximate sampler can be obtained by choosing a ground state uniformly at random and then running the Glauber dynamics starting from there. In a similar direction, in \cite[Theorem 13]{CheGaGoPeSteVi21} it was proven that restricting the Glauber dynamics to an appropriate portion of the state space is mixing polynomially fast for bounded degree graphs when $\lam$ is polynomially large in $\Delta$.

Here, we propose a Markov chain that takes advantage of the symmetry of the bipartite graph and injects a global move in the Glauber dynamics which flips majority-odd to majority-even sets and vice-versa. Hence, if that is indeed the only bottleneck, it should mix fast(er). We only state it for the hard-core model on the hypercube.

For a set $S \subseteq V(Q^d)$, let $S+e_1$ denote the vertex set obtained from $S$ by flipping the first coordinate of each $v \in S$. The \textit{Glauber dynamics with flips} is defined as follows. Given that we are at state $X_t:=I$ in time $t$, the state $X_{t+1}$ is determined as follows.
\begin{enumerate}
    \item Choose a vertex $v$ uniformly at random.
    \item Let $S':=S\cup\{v\}$ with probability $\frac{\lam}{1+\lam}$ and $S':=S$ with probability $\frac{1}{1+\lam}$.
    \item If $S'$ is an independent set, let $S'':=S'$. Else, let $S'':=S$.
    \item Let $X_{t+1}:=S''$ or $X_{t+1}:=S''+e_1$ each with probability $1/2$.
\end{enumerate}
In other words, the chain does a step of the Glauber dynamics and then with probability $1/2$ ``flips'' the independent set to the other side of the bipartition. This corresponds to considering the chain with transition matrix $\frac{1}{2}(I+I_\sigma)P$, where $P$ is the transition matrix of the Glauber dynamics and $I_\sigma$ is the permutation matrix, with columns indexed by the elements of the state space, corresponding to the operation $S\to S+e_1$. Note that the stationary distribution $\pi$ is still the Gibbs measure $\mu_{\lam,\beta}$: the automorphism $v\to v+e_1$ of $Q^d$ preserves the Gibbs measure of sets ($\pi(S)=\pi(S+e_1)$ for all sets $S$), and hence $\pi I_\sigma=\pi$ and $\pi\cdot\frac{1}{2}(I+I_\sigma)P=\pi$.

Specializing to the hypercube $Q^d$, where $n=2^d$, the result of \cite{CheGaGoPeSteVi21} only yields a quasipolynomial bound on the mixing time of the restricted Glauber dynamics for $\lam\geq n^{\Om(\log\log n)}$ , so we cannot disregard the possibility that another bottleneck is created. Hence, we pose the following question for the hypercube.
\begin{question}\label{ques:global_glauber}
    Does the Glauber dynamics with flips on the hypercube mix in polynomial time when $\lam$ is bounded?
\end{question}
The global move we imposed uses the symmetry of the graph in an essential way, exploiting an automorphism of the graph that maps the sides of the bipartition to one another to break the bottleneck. One may use another such automorphism (e.g., adding one to some other coordinate), or even choose one according to some distribution (e.g., choosing the coordinate to flip uniformly at random). 

On the other hand, it is unknown whether a polynomial-time approximate sampler or an FPRAS for the hypercube exist, and potentially the problem could even be hard on this particular instance. Therefore, we pose the following question.
\begin{question}\label{ques:hypercube_bis}
    Is there an FPRAS for the number of independent sets in the hypercube?
\end{question}

\section*{Acknowledgments}
The authors would like to thank Will Perkins for useful discussions. This research was supported in part by the Austrian Science Fund (FWF) [10.55776/F1002].
For open access purposes, the authors have applied a CC BY public copyright license to any author accepted manuscript version arising from this submission.

\printbibliography

\newpage
\appendix

\section{Equivalence to the Ising model}\label{sec:ising}

In this appendix, we present an elementary proof of the equivalence of the spin model studied in this paper to the classical version of the antiferromagnatic Ising model. More general results on the equivalence of $2$-spin systems can be found in \cite{SuSly12}.
Recall that given a graph $G=(V, E)$, the model we study samples a subset $S\subseteq V$ (which can be seen as a $\{0,1\}$-configuration in the obvious way) with probability proportional to $\lam^{|S|} e^{-\beta |E(S)|}$.
In this section, we will call this model the lattice gas model.
In the classical Ising model, a configuration is a spin configuration $\sigma \in \{-1,+1\}^{V}$ and its weight is given by
\[
\exp\left({J \sum_{uv \in E} \sigma(u) \sigma(v) + h \sum_{w \in V} \sigma(w)}\right),
\]
where $J \in \mathbb{R}$ is the interaction coefficient and $h \in \mathbb{R}$ is an external magnetic field.
The negative of the expression in the exponent is called the Hamiltonian. The Ising model is called \textit{ferromagnetic} if $J > 0$ and called \textit{antiferromagnetic} if $J < 0$.

\begin{proposition}
    On a $d$-regular graph $G$, the lattice gas model with fugacity $\lam > 0$ and inverse temperature $\beta$ is equivalent to the classical Ising model with $J=-\frac{\beta}{4}$ and $h=\frac{\log \lam}{2} - \frac{\beta d}{4}$.
\end{proposition}

\begin{proof}
    Let $S \subseteq V$ be a configuration of the lattice gas model.
    The Hamiltonian is given by
    \[
    H(S)=\beta |E(S)| - (\log \lam)|S| = \beta \sum_{uv \in E} \mathbbm{1}_{u \in S} \mathbbm{1}_{v \in S} - (\log \lam) \sum_{w \in S} \mathbbm{1}_{w \in S}.
    \]
    Now set $\sigma(v)\coloneqq 2\mathbbm{1}_{v \in S} -1$. This defines a bijection to configurations of the classical Ising model, where $\sigma(v)=1$ if and only if $v \in S$. Substituting $\mathbbm{1}_{v \in S}= \frac{1}{2}(\sigma(v)+1)$, we obtain
    \[
    H(S)= \frac{\beta}{4} \sum_{uv \in E}\left( \sigma(u) \sigma(v) + \sigma(u)+\sigma(v) + 1\right) - \frac{\log \lam}{2} \sum_{w \in V} (\sigma(w)+1).
    \]
    Since $G$ is $d$-regular, we obtain that 
    \[
    H(S)= \frac{\beta}{4} \sum_{uv \in E} \sigma(u) \sigma(v) + \left( \frac{\beta d}{4} - \frac{\log \lam}{2} \right) \sum_{w \in V} \sigma(w) + \left(\frac{\beta d}{8} - \frac{\log \lam}{2} \right)n,
    \]
    where $n=|V|$. The last term is independent of the configuration, and thus does not affect the resulting Gibbs measure.
\end{proof}

\section{The Container Lemma}\label{sec:appendix_containers}
In this appendix, we prove \Cref{lem:middle}, which gives an upper bound on the weight $\tilde{\omega}(\cS_{big, \cD})$ of $\cS_{big, \cD}$ using graph containers. 
We assume that we work with a graph $G\in\cH(n,d,\codeg,\ex)$ with constant maximum codegree $\codeg \geq 1$ and with constant $0 \leq \ex<2$. The following definitions are crucial for the construction of containers.

\begin{definition}\label{def:phi_approx}
    Let $1\leq \varphi\leq d-1$ and let $A\subset \cD$. A \textit{$\varphi$-approximation} of $A$ is a set $F'\subseteq \overline{\cD}$ such that
    \begin{equation*}
        N(A)^{\varphi}\subseteq F' \subseteq N(A)\qquad \text{and} \qquad N(F')\supseteq [A],
    \end{equation*}
    where $N(A)^{\varphi}:=\{v\in N(A): d_{[A]}(v)> \varphi\}$.
\end{definition}

\begin{definition}\label{def:psi_approx}
Let $1 \le \psi \le d - 1$. A \textit{$\psi$-approximation} $(F,H)$ is a pair of vertex sets $F \subseteq \overline{\cD}, H \subseteq \cD$ such that for some set $A \subseteq \cD$, 
    \begin{equation*}
        F\subseteq N(A)  \ \ \text{and} \ \ H\supseteq [A],
    \end{equation*}
    \begin{equation*}
        d_F(u)\geq d(u)-\psi, \quad \forall u\in H,
    \end{equation*}
    \begin{equation*}
        d_{\cO\setminus H}(v)\geq d(v)-\psi, \quad \forall v\in \cE\setminus F.
    \end{equation*}
    In this case, we say that $(F, H)$ approximates $A$, and we denote this relation by $A\approx (F,H)$.
\end{definition}
What follows is an adaptation of Galvin \cite{galvin2019independent} to the case where the considered sets are not $2$-linked. Of course, this generalization is not valid in all ranges of the parameters, but it is for the large sets we consider. Recall that for $\cD \in \{\cO, \cE\}$ and $a, b \in \mathbb{N}$,
\begin{equation*}
    \cG_{\cD}(a, b) \coloneqq \{A \subseteq \cD : |[A]|=a, |N(A)|=b\}.
\end{equation*}
\vspace{-0.7cm}
\begin{lemma} \label{lem:container_family}
    Let $t:=b-a$. If $b\geq n/d$, then there exists a family $\cA$ of pairs $(F,H)$ with
    \begin{equation*}
    |\cA| \le \exp\left[O\left(\frac{b\log^2d}{d^2} + \frac{t \log^2 d}{d} \right)\right],
    \end{equation*}
such that every $A\in \cG(a,b)$ has a $d/2$-approximation $(F,H)$ in $\cA$.
\end{lemma}
The reader familiar with the method will notice that there is a factor of $|\cD|$ missing in front of the upper bound on $|\cA|$. This is not an improvement but is instead dominated by the exponential term. The container family $\cA$ is obtained in two steps. First a rough $\varphi$-approximation is constructed. Then this is refined into a $\psi$-approximation. More precisely, \Cref{lem:container_family} is implied by the combination of the following two lemmas.
\begin{lemma}\label{lem:phi_family}
    There exists a family $\cV=\cV(\varphi)$ such that every set $A\in\cG_{\cD}(a,b)$ has a $\varphi$-approximation with $\varphi=d/2$ in $\cV$, and
    \begin{equation*}
        |\cV|=\exp\left[O\left(\left(\frac{b\log d}{d^2} + \frac{t \log d}{d} \right)\log\left(\frac{nd^2}{b \log d}\right)+\frac{t\log^2 d}{d}\right)\right].
    \end{equation*}
\end{lemma}
Note that the bound on $|\cV|$ gives an exponential term in $n$ unless $b$ is very large, which is the case for us. 

\begin{lemma}\label{lem:psi_family}
    For any $F'\in\cV(\phi)$, there exists a family $\cW=\cW(F')$ such that every $A\in\cG(a,b)$ whose $\varphi$-approximation is $F'$ has a $\psi$-approximation in $\cW$ where $\varphi=\psi=d/2$, and
    \begin{equation*}
        |\cW|=\exp\left(O\left(\frac{t\log d}{d}\right)\right).
    \end{equation*}
\end{lemma}

\Cref{lem:container_family} follows immediately from Lemmas \ref{lem:phi_family} and \ref{lem:psi_family} using $|\cA|\leq|\cV|\cdot|\cW|$ and $b\geq n/d$. The proof of \Cref{lem:psi_family} is identical to that in \cite{galvin2019independent}, modulo changing the notation accordingly and setting $\varphi=\psi=d/2$.
We will only present the proof of \Cref{lem:phi_family}, closely following \cite[Section 5.3]{galvin2019independent}. Note that the notation in this section does not interfere with the notation in the main body of the paper.

\begin{proof}[Proof of \Cref{lem:phi_family}]
    Fix some $A\in\cG_{\cD}(a,b)$ and let $p \coloneqq \frac{2C \log d}{d^2}$, where $C$ is a constant to be determined later. Choose a $p$-random subset $S$ of $N(A)$. Then $|S| \sim \text{Bin}(b, p)$, so that
    \begin{equation}\label{eqn:S_exp_size}
    \Ex{|S|}=bp = \frac{2C \log d}{d^2} b.
    \end{equation}
    Let $\Omega(S) \coloneqq E(S, \cD \setminus [A])$. Since $|E(N(A), \cD \setminus [A])|=td$, we conclude that
    \begin{equation}\label{eqn:S_escape}
        \Ex{|\Omega(S)|} = t dp = \frac{2C t \log d}{d}.
    \end{equation}
    Let $S' \coloneqq N(A)^{\varphi} \setminus N(N_{[A]}(S))$. Following the calculation of (28) in \cite{galvin2019independent}, one may conclude that
    \begin{equation}\label{eqn:S'_size}
        \Ex{|S'|} \leq  \frac{b}{d^2},
    \end{equation}
    when $C$ is chosen appropriately, solely depending on the maximum codegree of the graph. Using Markov's inequality and equations \eqref{eqn:S_exp_size}, \eqref{eqn:S_escape}, and \eqref{eqn:S'_size}, there exists at least one $T_0 \subseteq N(A)$ satisfying
    \begin{equation*}
        |T_0| \leq \frac{4C b \log d}{d^2} \quad \text{and} \quad
        |\Omega(T_0)| \leq \frac{8Ct\log d}{d} \quad \text{and} \quad
        |T_0'| \leq \frac{4b}{d^2}.
    \end{equation*}
    Setting $L \coloneqq N(N_{[A]}(T_0)) \cup T_0'$, let $T_1 \subseteq N(A) \setminus L$ be a cover of minimum size of $[A] \setminus N(L)$ in the graph induced by $(N(A) \setminus L) \cup ([A] \setminus N(L))$. Then $F':=L\cup T_1$ is a $\varphi$-approximation of $A$. 
    
    Following the argument after the proof of Claim 5.7 in \cite{galvin2019independent}, we have 
    \begin{align*}
        |T_1| \leq \frac{3t \log d}{d-\varphi} = \frac{6t \log d}{d}.    
    \end{align*}
   Set $T \coloneqq T_0 \cup T_0' \cup T_1$  and denote
    \begin{align*}
        T_{bound} \coloneqq \frac{4C b \log d}{d^2} + \frac{4b}{d^2} + \frac{6t \log d}{d}.
    \end{align*}
    Since $|T| \leq T_{bound}$, the number of choices for $T$ is at most
    \begin{align} \label{eq:bound_sizeT}
        \binom{n}{|T|} \leq\binom{n}{T_{bound}} \leq \exp\left(O\left(\frac{b\log d}{d^2} + \frac{t \log d}{d} \right)\log\left(\frac{nd^2}{b \log d}\right)\right),
    \end{align}
    where we have used \eqref{eqn:binom_estimate}, the fact that $k\log(en/k)$ is increasing for $k\leq n$, and that if $k\geq k'$ then $k\log(en/k)\leq k\log(en/k')$. Note that here lies the main difference to the proof from \cite{galvin2019independent}, which uses a finer bound on the number of choices for $T$ stemming from the fact that the corresponding set in \cite{galvin2019independent} is $8$-linked. Nonetheless, the crude bound in \eqref{eq:bound_sizeT} will suffice for our purposes.

    Once $T$ is chosen, there are at most $2^{|T|}$ choices for each $T_0, T_1 \subseteq T$, while $T_0'$ is determined by $T_0$. Given that $|T_0|\leq \frac{4C b \log d}{d^2}$ and that $b/t=O(d^{\ex})=O(d^2)$, there are at most 
    $$\binom{\frac{4C b \log d}{d}}{\leq \frac{6Ct \log d}{d}}\overset{\eqref{eqn:binom_estimate}}{\leq} \exp\left(O\left(\frac{t\log^2 d}{d}\right)\right)$$
    choices for $\Omega(T_0)$. Altogether, the number of choices for $(T_0, T_0', T_1, \Omega(T_0))$ (and hence $F'$, which is completely determined by these sets) is at most
    $$2^{T_{bound}}\exp\left[O\left(\left(\frac{b\log d}{d^2} + \frac{t \log d}{d} \right)\log\left(\frac{nd^2}{b \log d}\right)+\frac{t\log^2 d}{d}\right)\right].$$
    As $2^{T_{bound}}$ is clearly dominated by the first term of the exponential, the proof of \Cref{lem:phi_family} is finished.
\end{proof}
Now the proof of \Cref{lem:container_main} is identical to the proof of \cite[Lemma 4.2]{GeKaSaWdo25}, namely by combining \cite[Lemma 4.5]{GeKaSaWdo25} with \Cref{lem:container_family} and noting that $b\geq a\geq f(\lam)\frac{n}{2} \geq n/d$.

\section{Proof of \Cref{lem:Kotecky}}\label{sec:appendix_tails}
In this appendix, we sketch the proof of \Cref{lem:Kotecky}, describing how many terms are needed in the computation of the cluster expansion to get within $\epsilon/16$ of the desired partition function. Define 
\begin{align*}
    \tilde \alpha \coloneqq \frac{1+\lambda}{1+\lambda e^{-\beta}}
\end{align*}
and note that $1 < \tilde \alpha \le 1+\lambda \le 1+\lam_0$ and $\log \tilde \alpha \ge \frac{\alpha}{1+\lambda}$. In \cite[Proof of Lemma 3.2, see Section 8]{GeKaSaWdo25}, what is actually shown is that
\begin{equation*}
    |L_{\cD,\geq k}|\leq nd^{-(\ex+1)}\exp{(-\tilde g(k))},
\end{equation*}
where the function $\tilde g$ is given by
\begin{align*}
\tilde{g}(k) &\coloneqq     \left\{\begin{array} {c@{\quad \textup{if} \quad}l}
       (d k -\codeg k^2)\log\ta - (\ex+7) k \log d & k \leq \frac{d}{\log \log d}, \\[+1ex]
      \frac{\sqrt{d} k}{2} \log\ta & \frac{d}{\log \log d} < k \le d^3 \log n, \\[+1ex]
      \frac{k}{d^{\ex+1}} & k > d^3 \log n.
    \end{array}\right.
\end{align*}
We aim to find $k_0$ such that $|L_{\cD,\geq k_0}|\leq \epsilon/16$, that is,
\begin{equation}\label{eqn:k0}
    \tilde g(k_0)\geq \log\left(\frac{16n}{\epsilon d^{\ex+1}}\right).
\end{equation}
Note that in the first case, $\tilde{g}$ does not depend on $n$ while the right-hand side of \eqref{eqn:k0} grows with $n$, such that \eqref{eqn:k0} is never fulfilled. In the second case, we note that $\tilde{g}$ is increasing in $k$. Thus, if \eqref{eqn:k0} is fulfilled for some $k$ in this regime, then it is also fulfilled for $k_0=\lfloor d^3 \log n \rfloor$.
In the third case, \eqref{eqn:k0} rearranges to
\begin{equation*}
    k_0\geq d^{\ex+1}\log\left(\frac{16n}{\epsilon d^{\ex+1}}\right).
\end{equation*}
Thus, in this regime we may take $k_0=\max\{\lfloor d^3 \log n +1 \rfloor, \lceil d^{\ex+1}\log\left(\frac{16n}{\epsilon d^{\ex+1}}\right) \rceil \}$. This completes the proof of \Cref{lem:Kotecky}.
\end{document}